\documentclass[preprint,11pt]{elsarticle}

 \journal{ }
\usepackage{graphicx}
\usepackage{pifont,latexsym,ifthen,amsthm,rotating,calc,textcase,booktabs}
\usepackage{amsfonts,amssymb,amsbsy,amsmath}
\newtheorem{theorem}{Theorem}[section]
\newtheorem{lemma}[theorem]{Lemma}
\newtheorem{corollary}[theorem]{Corollary}

\newtheorem{proposition}[theorem]{Proposition}

\newtheorem{problem}{Problem}
\newtheorem{main theorem}{Theorem}
\usepackage{natbib}

\usepackage{todonotes}
\usepackage{float}
\usepackage{subfigure}
\usepackage{hyperref}
\hypersetup{
    colorlinks=true, 
    linktoc=all,     
    linkcolor=blue,  
    citecolor=blue
}

\begin{document}



\begin{frontmatter}
\title{Existence of global solutions to the nonlocal mKdV equation on the line}


\author[inst2]{Anran Liu}

\author[inst2]{Engui Fan$^{*,}$  }

\address[inst2]{ School of Mathematical Sciences and Key Laboratory of Mathematics for Nonlinear Science, Fudan University, Shanghai,200433, China\\
* Corresponding author and e-mail address: faneg@fudan.edu.cn  }





\begin{abstract}
  In this paper,  we address   the existence of global solutions to the Cauchy problem for the integrable nonlocal   modified Korteweg-de vries (nonlocal mKdV)
   equation with the initial data $u_0 \in H^{3}(\mathbb{R}) \cap H^{1,1}(\mathbb{R}) $ with the $L^1(\mathbb{R})$  small-norm assumption.
A  Lipschitz  $L^2$-bijection map  between  potential   and reflection coefficient
is established  by using inverse scattering method based
on a  Riemann-Hilbert problem associated with the Cauchy problem.
The map from initial potential to reflection coefficient is obtained  in  direct scattering transform.
The inverse scattering transform goes back to the   map  from scattering coefficient  to  potential  by  applying
 the reconstruction formula and Cauchy
integral operator.  The bijective relation naturally yields
the existence of  a global solutions    in a Sobolev space  $H^3(\mathbb{R})\cap H^{1,1}(\mathbb{R})$
to the Cauchy problem.

\end{abstract}

\end{frontmatter}
\tableofcontents

\section{Introduction and main results}
In this paper, we  establish   the global existence of solutions to
 the Cauchy problem for 
   the nonlocal    mKdV   equation  
\begin{align}
& u_t(x,t)+   u_{xxx}(x,t) +6\sigma u(x,t)u(-x,-t)  u_x(x,t)=0,  \label{mkdv1}\\
 & u(x,0)=u_0(x),\label{mkdv2}
\end{align}
where   $u_0 \in H^{3}(\mathbb{R}) \cap H^{1,1}(\mathbb{R}) $ and $\sigma=\pm1$ denote the focusing  and defocusing cases,
respectively.

The nonlocal mKdV equation \eqref{mkdv1}, introduced in \cite{nmkdv1,nmkdv2},
  can be regarded as the integrable nonlocal extension of the    classical  mKdV equation
\begin{equation}
u_t(x,t)+u_{xxx}(x,t)+6 \sigma u^{2}(x,t)u_{x}(x,t)=0,\label{lmkdv}
\end{equation}
by replacing $u^{2}(x,t)$ with the \emph{PT}-symmetric term $u(x,t)u(-x,-t)$  \cite{PT}.
In physical application, the nonlocal mKdV equation  \eqref{mkdv1} possesses the shifted parity and delayed time reversal symmetry,
and thus it can be related to the Alice-Bob system \cite{10}. For instance, a special approximate solution of the nonlocal mKdV was applied to theoretically capture the salient features of two correlated dipole blocking events in atmospheric dynamical
systems \cite{11}.

There is much work on the study of various mathematical properties for the nonlocal mKdV equation \eqref{mkdv1}.
 The $N$-soliton solutions for  the  nonlocal mKdV equation \eqref{mkdv1}
    with zero boundary conditions  were constructed by using  the  Darboux transformation   and the inverse scattering  transform   respectively  \cite{Darboux,IST}.
Further the    Riemann-Hilbert (RH) method was   used  to construct N-soliton solutions    for
the nonlocal   mKdV  equations \eqref{mkdv1} with  nonzero boundary conditions   \cite{ZY}.
 The long-time asymptotics for the nonlocal   mKdV equation \eqref{mkdv1}   with decaying initial data  was   investigated   in  \cite{HF}  via  the nonlinear  steepest-descent method
 developed by  Deift and Zhou   \cite{DZ1993}.
  Recently, we obtained
   the long time asymptotic behavior for the Cauchy problem of the  nonlocal mKdV equation \eqref{mkdv1} with nonzero initial data in the solitonic regions
    by using the $\bar{\partial}$-steepest-descent method \cite{ZF1,ZF2}.  
    This method, introduced   by   McLaughlin and Miller   \cite{MM1,MM2},     has been extensively implemented
in  the long-time asymptotic analysis  and  the soliton resolution conjecture of  some integrable systems  \cite{fNLS,Liu3,LJQ,YF1,YYLmch}.
 However, the   existence  global solutions to  the Cauchy problem   (\ref{mkdv1})-(\ref{mkdv2}) for  the  nonlocal mKdV equation   is   still
     unknown  to our best knowledge.   A technical  difficulty  to apply partial differential analytical technique for
     proving  global   existence of   the nonlocal   mKdV equation \eqref{mkdv1}
 comes from the fact that
 the   mass and energy conservation laws to  the   equation \eqref{mkdv1}
 are in  the form
\begin{align}
&I_0=\int_{\mathbb{R}} q(x,t)\overline{q(-x, -t)}dx, \nonumber\\
&I_1=\int_{\mathbb{R}} [ q_x(x,t)\overline{q_x(-x, -t)} - \sigma q^2(-x,t) \overline{q^2(-x,-t)} ]dx,\nonumber
\end{align}
   do not preserve any reasonable norm and  may  be negative.
 In contrast with this,  the  mass and energy conservation laws  of  the   classical mKdV equation (\ref{lmkdv})
  allows to obtain a priory estimates for
 establishing  a unique global  solution.

The main purpose in the present paper is to  overcome this  difficulty and establish   the global existence of solutions to
 the Cauchy problem (\ref{mkdv1})-(\ref{mkdv2})  in
 an appropriate Sobolev space by applying  the  inverse scattering theory.
Our   principal result  is  now stated as follows.

\begin{main theorem} \label{th1}
Let   the initial data  $u_0 \in H^3(\mathbb{R}) \cap H^{1,1}(\mathbb{R})$ such that the spectral problem
     (\ref{lax1}) admits no eigenvalues or resonances. Then

\begin{itemize}

 \item  There exists a  $L^2$-bijection map   between  the  potential $u$
     and reflection coefficients $r_{1,2}$,
\begin{equation}
H^3(\mathbb{R}) \cap H^{1,1}(\mathbb{R})\ni u \mapsto r_{1,2} \in H^{1,1}(\mathbb{R}) \cap L^{2,3}(\mathbb{R}), \label{1.3}
\end{equation}
which is Lipschitz continuous.

\item There exists a unique global solution $
u \in C([0,\infty),H^3(\mathbb{R}) \cap H^{1,1}(\mathbb{R}))
$ to the Cauchy problem (\ref{mkdv1})-(\ref{mkdv2}).
Furthermore, the map
$$
 H^3(\mathbb{R}) \cap H^{1,1}(\mathbb{R}) \ni u_{0}  \mapsto u  \in C\left([0, \infty); H^3(\mathbb{R}) \cap H^{1,1}(\mathbb{R})\right)
$$
is Lipschitz continuous.

\end{itemize}

\end{main theorem}

A key  in proving the above   result   is   to  establish
a Lipschitz  $L^2$-bijection  (\ref{1.3})    between  solution  and scattering coefficient by using inverse scattering method \cite{zhou1,zhou2,Pelinovsky} .
  The $L^2$-bijection (\ref{1.3})  implies that     global well-posdeness of the Cauchy problem (\ref{mkdv1})-(\ref{mkdv2})
in the space $H^{3}(\mathbb{R})\cap H^{1,1}(\mathbb{R})$.

The structure of the paper is as follows. In Section \ref{sec2}, we  focus on    the direct scattering
 transform to the Cauchy problem (\ref{mkdv1})-(\ref{mkdv2}).     We especially  establish  the Lipschitz continuous
  maps  from the initial data to the Jost function and  the reflection coefficient.
 In Section \ref{sec3}, we carry out the inverse scattering transform to  set up a   RH  problem associated with the Cauchy problem (\ref{mkdv1})-(\ref{mkdv2}),
 Further the  solvability of the RH problem  is   shown.
  In Section \ref{sec4}, we reconstruct  and estimate the potential  from  the solutions of the RH
  problem  on positive half line  $\mathbb{R}^{+}$ and negative  half line $\mathbb{R}^{-}$ respectively.
 We further  establish a Lipschitz continuous mapping from  the reflection coefficients to the potentials. In Section \ref{sec5}, we perform the time evolution
 of the reflection coefficients and the RH problem.  Then, we prove that there exists   a  unique  global solution to the initial value problem (\ref{mkdv1})-(\ref{mkdv2})
 of the nonlocal mKdV equation in the space $H^{3}(\mathbb{R})\cap H^{1,1}(\mathbb{R})$.

\section{Direct  scattering transforms} \label{sec2}

In this section, we  state some  main
results on  the  direct  scattering transform associated with the Cauchy problem
(\ref{mkdv1})-(\ref{mkdv2}).   The details   can be found in
\cite{nmkdv1,nmkdv2}.

\subsection{Lipschitz continuity of the Jost functions }

The nonlocal mKdV equation  (\ref{mkdv1})   admits   the  Lax pair
\begin{align}
&\psi_x -iz \sigma_{3}\psi = Q  \psi, \label{lax1}\\
&\psi_t -4iz^3\sigma_{3}\psi =(4z^{2}Q-2iz(Q_{x}-Q^{2})\sigma_{3}+2Q^3-Q_{xx} ) \psi,\label{lax2}
\end{align}
where
\begin{align}
&Q =\begin{pmatrix}
0&u(x,t) \\
- \sigma u(-x,-t) &0
\end{pmatrix}, \ \
\sigma_3=\begin{pmatrix}
1&0\\
0&-1
\end{pmatrix}.\nonumber
\end{align}

Define the Jost functions    $\psi^{\pm}(x,z)$   to the spectral problem (\ref{lax1})  with  the following boundary conditions
$$\psi^{\pm}(x,z) \sim  e^{izx\sigma_{3}}, \ x\to \pm\infty.$$
Making a transformation
$$
m^\pm(x,z) =\psi^\pm(x,z)  e^{-izx\sigma_3},
$$
then
 \begin{equation}
 \lim\limits_{x\to\pm\infty} m^{\pm}(x,z)=I,\nonumber
 \label{m}\end{equation}
 and $m^\pm(x,z)$  satisfy  the Voterra integral equations
\begin{align}
m^{\pm}(x,z)=I+\int_{\pm \infty}^{x}e^{-iz(y-x)\operatorname{ad}\sigma_{3}}Qm^{\pm}(y,z)dy,\label{m_{1}}
\end{align}
where    $e^{\operatorname{ad} \sigma_{3}} A := e^{ \sigma_{3}}A  e^{-\sigma_{3}}$.

Denote $m^{\pm}(x,z)=[m_{1}^{\pm}(x,z),m_{2}^{\pm}(x,z)]$ .
From symmetry of Lax pair, we can get
\begin{equation}
 m_{1}^{\pm}(x,z)=\sigma\Lambda
 \overline{m_{2}^{\mp}(-x,-\overline{z}}),\quad
  m_{2}^{\pm}(x,z)=\Lambda
 \overline{m_{1}^{\mp}(-x,-\overline{z}}),
\label{symmetry}\end{equation}
where
$$
\Lambda=\begin{pmatrix}
0&\sigma\\
1&0
\end{pmatrix}.
$$
It can be further  shown   that functions $m^{+}_{1}(x,z)$ and  $m^{-}_{2}(x,z)$  are analytic in  $z\in \mathbb{C}^{+}$,
  whereas the functions $m^{+}_{2}(x,z)$ and $m^{-}_{1}(x,z)$ are analytic in  $z\in \mathbb{C}^{-}$.
There is a matrix $S(z)$ satisfying
\begin{equation}
\label{1}
m^{+}(x,z)=m^{-}(x,z) e^{izx\sigma_{3}} S(z),
 \end{equation}
 where
$$
S(z)=\begin{pmatrix}
a(z)&c(z)\\
b(z)&d(z)
\end{pmatrix}.
$$

From (\ref{1}), we deduce that
\begin{align}
&a(z)= \det[m_{1}^{+}(x,z),m_{2}^{-}(x,z)]\label{a}\\
&d(z)= \det[m_{1}^{-}(x,z),m_{2}^{+}(x,z)]\label{d}\\
&b(z) =\det[m_{1}^{-}(x,z),m_{1}^{+}(x,z)]e^{-2izx}\label{b}.
\end{align}
It can be shown that   $a(z)$  is  analytic in $C^{+}$ and  $a(z)\to 1$ as $z\to\infty$ in $\overline{C^{+}}$ while $d(z)$  is  analytic in $C^{-}$ and  $d(z)\to 1$ as $z\to\infty$ in $\overline{C^{-}}$.

From $(\ref{symmetry})$ we can get the scattering coefficients satisfy the following symmetry:
$$
a(z)=\bar{a}(-\bar{z}),d(z)=\overline{d(-\bar{z})},c(z)=-\sigma\overline{b(-\bar{z})}.
$$
We define the reflection coefficient:
$$
r_{1}(z)=b(z) / a(z),r_{2}(z)=c(z) / d(z)=-\sigma\overline{b(-z)} / d(z) \quad z \in \mathbb{R}.
$$

The determinant of $S(z)$ is
$$
a(z)d(z)+\sigma b(z)\overline{b(-\bar{z})}=1.
$$
In the follows, we  prove the existence   of $m^{\pm}(x,z)$.  For
$$
f(\cdot,z)=
(
f_{1}(\cdot,z),
f_{2}(\cdot,z)
)^{\top}\in L^{\infty}(\mathbb{R}),
$$
 define
\begin{equation}
(K_{u}f)(x,z):=-\int_{x}^{\infty}
 {\rm diag} (
1, e^{2iz(y-x))}  Q (y) f(y,z)dy,
\label{2.24}
\end{equation}
then the equation (\ref{m_{1}}) can be written as
\begin{equation}
(I-K_{u}) m_{1}^{+}(x,z)=e_{1}.
\label{2}
\end{equation}
Then the   operators of $K_{u}$  has the following property.


\begin{lemma}
Let $u\in L^{1}(\mathbb{R})$, for fixed $z\in \overline{\mathbb{C}^{+}}$,  $I -K_{u}$ is an invertible operator in  $L^{\infty}(\mathbb{R})$.
\label{lemma1}\end{lemma}
\begin{proof}
Notice that $z\in \overline{\mathbb{C}^{+}}$ and
\begin{equation}
\begin{aligned}
| K_{u} f(x,z)|&\leq |(K_{u}f)_{1}(x,z)|+|(K_{u}f)_{2}(x,z)|\\&
=|\int_{x}^{\infty}u(-y)f_{1}dy|+|\int_{x}^{\infty}u(y)e^{2iz(y-x)}f_2dy|\\&
\leq\|u\|_{L^{1}(\mathbb{R})}\|f\|_{L^{\infty}(\mathbb{R})},
\end{aligned}\nonumber
\end{equation}
which implies  that $K_{u}$ is a bounded operator in $L^{\infty}(\mathbb{R})$  for any fixed $z\in\overline{\mathbb{C}^{+}}$.

Similar to the analysis described above, we have for $n\geq1$,
\begin{equation}
|K_{u}^{n}f|\leq(1/n!)\|u\|_{L^{1}(\mathbb{R})}^{n}\|f\|_{L^{\infty}(\mathbb{R})}.
\end{equation}

From above analysis we know for any fixed $z\in\overline{\mathbb{C}^{+}}$, $K_{u}^{n}$ is a bounded operator in $L^{\infty}(\mathbb{R})$ and 
$$\|K_{u}^{n}\|_{L^{\infty}(\mathbb{R})\to L^{\infty}(\mathbb{R})}\leq (1/n!)\|u\|_{L^{1}(R)}^{n},$$
 which yields $1-K_{u}$ is an invertible operator in  $L^{\infty}(\mathbb{R})$. Moreover 
 $$\|(1-K_{u})^{-1}\|_{L^{\infty}(\mathbb{R})\to L^{\infty}(\mathbb{R})}\leq e^{\|u\|_{L^{1}(\mathbb{R})}}.$$
\end{proof}

Next, we study the asymptotics  of  the Jost functions $m^{\pm}(x,z)$.
\begin{lemma}
If $u\in H^{3}(\mathbb{R})\cap H^{1,1}(\mathbb{R})$, then  as  $|{\rm Im} z|\rightarrow \infty$,  for every $x\in\mathbb{ R}$,
\begin{align}
m_{1}^{\pm}(x;z)=&e_{1}+p_1^\pm (x)(2iz)^{-1}+q_1^\pm  (x) (2iz)^{-2}+g_1^\pm  (x) (2iz)^{-3}+\mathcal{O}(z^{-4}),\nonumber\\
m_{2}^{\pm}(x;z)=&e_{2}+p_2^\pm (x)(2iz)^{-1}+q_2^\pm(x)(2iz)^{-2}+\mathcal{O}(z^{-3}),\nonumber
\end{align}
where
\begin{align}
&p_1^\pm(x) = \left[\sigma\int_{x}^{\pm\infty}u(-y)u(y)dy,\   -\sigma u(-x) \right]^{\top},\\\nonumber
&q_1^\pm (x) =\begin{bmatrix}
\sigma\int_{x}^{\pm\infty}\partial u(-y)u(y)dy+\int_{x}^{\pm\infty}u(-x_{1})u(x_{1})\int_{x_{2}}^{\pm\infty}u(-x_{2})u(x_{2})\\
-\sigma\partial u(-x)-u(-x)\int_{x}^{\pm\infty}u(-y)u(y)dy
\end{bmatrix},\\\nonumber
&g_1^\pm (x) =\\\nonumber
&\begin{bmatrix}
\sigma\int_{x}^{\pm\infty}u(y)\partial^{2} u(-y)dy+\int_{x}^{\pm\infty}u(x_{1})\partial u(-x_{1})\int_{x_{1}}^{\pm\infty}u(-x_{2})u(x_{2})+\int_{x}^{\pm\infty}u(-y)^{2}u(y)^{2}dy\\
+\sigma\int_{x}^{\pm\infty}u(-x_{1})u(x_{1})\int_{x_{1}}^{\pm\infty}u(-x_{2})u(x_{2})\int_{x_{2}}^{\pm\infty}u(-x_{3})u(x_{3})+\int_{x}^{\pm\infty}u(-x_{1})u(x_{1})\int_{x_{1}}^{\pm\infty}u(x_{2})\partial u(-x_{2})\\ \\
-\partial u(-x)\int_{x}^{\pm\infty}u(-y)u(y)dy-\sigma u(-x)\int_{x}^{\pm\infty}u(-x_{1})u(x_{1})\int_{x_{1}}^{\pm\infty}u(-x_{2})u(x_{2})\\
- u(-x)\int_{x}^{\pm\infty}\partial u(-y)u(y)dy-u^{2}(-x)u(x)-\sigma\partial^{2} u(-x)
\end{bmatrix},\\\nonumber
\label{2.36}\\
&p_2^\pm (x) = \left[-u(x), \sigma\int_{x}^{\pm\infty}u(y)u(-y)dy \right]^{\top}, \\\nonumber
&q_2^\pm (x) =\begin{bmatrix}
-\sigma\partial u(-x)-u(-x)\int_{x}^{\pm\infty}u(-y)u(y)dy\\
\sigma\int_{x}^{\pm\infty}\partial u(-y)u(y)dy+\int_{x}^{\pm\infty}u(-x_{1})u(x_{1})\int_{x_{2}}^{\pm\infty}u(-x_{2})u(x_{2})
\end{bmatrix}.
\label{2.37}\end{align}

\label{lemma2}\end{lemma}
\begin{proof}
We will only to prove the statement for $m_{1}^{\pm}(x;z)$,  while  $m_{2}^{\pm}(x;z)$ is similar to be proved.
Rewriting (\ref{m_{1}}) as the component form:
$$
m_{11}^{+}(x;z)=1-\int_{x}^{\infty}u(y)m_{21}^{+}(y;z)dy.
$$
$$
m_{21}^{+}(x;z)=\sigma\int_{x}^{\infty}e^{2iz(y-x)}u(-y)m_{11}^{+}(y;z)dy.
$$
We have proved that for every $x\in \mathbb{R}$, $m_{1}^{\pm}(x;z)$ is analytic in $z\in \mathbb{C}^{+}$. Noticing $u\in H^{3}(\mathbb{R})\cap H^{1,1}(\mathbb{R})\hookrightarrow L^\infty$ which yields (\ref{m_{1}})
is bounded for every $z\in \mathbb{C}^{+}$ and the integrand converges to $e_{1}$ as $|Imz|\rightarrow \infty$.
Integrating by part and recalling $u\in H^{3}(\mathbb{R})\cap H^{1,1}(\mathbb{R})$ we can get
\begin{align}
m_{11}^{+}(x;z)=&1+1/(2iz)\sigma\int_{x}^{\infty}u(y)u(-y)m_{11}^{+}dy+1/(2iz)^{2}\sigma\int_{x}^{\infty}u(y)\partial u(-y)m_{11}^{+}dy\nonumber\\
&+1/(2iz)^{2}\sigma\int_{x}^{\infty}u(y)u(-y)m_{21}^{+}dy+\mathcal{O}(z^{-3})\nonumber
\end{align}
and
$$
m_{21}^{+}(x;z)=-\sigma u(-x)m_{11}^{+}/(2iz)-\sigma(\partial u(-x)m_{11}^{+}-u(x)u(-x)m_{21}^{+})/(2iz)^{2}+\mathcal{O}(z^{-3}).
$$

Letting $|z|\rightarrow \infty$ and noticing $m_{1}^{+}\rightarrow e_{1}$ as $|Imz|\rightarrow \infty$ we get the expanding formula of $m_{1}^{\pm}(x;z)$.

\end{proof}

Similar to above analysis, we define $f(x;z)$ as a  binary function :
$$
f(x,z)=
(
f_{1}(x;z),
f_{2}(x;z) )^{\top}\in L_{x}^{\infty}(\mathbb{R})\otimes L_{z}^{2}(\mathbb{R}).
$$
$K_{u}$ $\partial_{z}K_{u}$  as   operators:
$$
 K_{u} f (x;z):=-\int_{x}^{\infty}
{\rm diag} (
1, e^{2iz(y-x))})  Q (y) f(y,z)dy,
$$
$$
(\partial_{z}K_{u}f)(x,z):=-\int_{x}^{\infty}
 {\rm diag} (
0, 2i(y-x)e^{2iz(y-x)})  Q (y) f(y,z)dy,
$$

We need to proof the following lemma first.

\begin{lemma}

If $w\in L^{2}(\mathbb{R})$, then
\begin{equation}
\sup_{x\in\rm{\mathbb{R}}}\left\| \int_{x}^\infty e^{-2iz(x-y)}w(y)dy \right\|_{L^2_z(\mathbb{R})}\leq \sqrt{\pi}\| w\|_{L^2(\mathbb{R})}.\label{5}
\end{equation}
If $w\in H^3(\mathbb{R})$, then for every $n=1,2,3$, we have
\begin{equation}
\sup_{x\in\rm{\mathbb{R}}}\left\| (2iz)^{n}\int_{x}^\infty e^{-2iz(x-y)}w(y)dy+\sum_{k=0}^{n-1}(-2iz)^{k+1}\partial_{x}^{k}w(x) \right\|_{L^2_z(\mathbb{R})}\leq \sqrt{\pi}\| \partial^{n}_x w\|_{L^2(\mathbb{R})}.
\label{6}
\end{equation}
If $w\in L^{2,1}(\mathbb{R})$, then for every $x_0\in \mathbb{R}^+$, we have
\begin{equation}
    \sup_{x\in (x_0,+\infty)}\left\|\langle x \rangle \int_{x}^\infty e^{-2iz(x-y)}w(y)dy \right\|_{L^2_z(\mathbb{R})}\leq \sqrt{\pi}\| w\|_{L^{2,1}(x_0,+\infty)}.
\label{8}
\end{equation}
Furthermore, if $w\in H^{1,1}(\mathbb{R})$, then for every $x_0\in \mathbb{R}^+$, we have
\begin{equation}
    \sup_{x\in (x_0,+\infty)}\left\|\langle x \rangle [(2iz)\int_{x}^\infty e^{-2iz(x-y)}w(y)dy+ w(x)]\right\|_{L^2_z(\mathbb{R})}\leq \sqrt{\pi}\| w\|_{H^{1,1}(x_0,+\infty)}
\label{45}
\end{equation}
and for every $x_0\in \mathbb{R}$, we have
\begin{equation}
    \sup_{x\in \mathbb{R}}\left\|(2iz)\int_{x}^\infty (y-x)e^{-2iz(x-y)}w(y)dy \right\|_{L^2_z(\mathbb{R})}\leq \sqrt{\pi}\| w\|_{H^{1,1}(\mathbb{R})},
\label{44}
\end{equation}
where $\langle x \rangle=(1+x^2)^{\frac 1 2}$.
\label{lemma3}\end{lemma}
\begin{proof}
The bounds (\ref{5}),(\ref{6}) for $n=1$, and  (\ref{8}) were given in \cite{Pelinovsky}. It remains to prove the estimate (\ref{6}) for $n=2,3$, (\ref{45}) and (\ref{44}). For every $x\in \mathbb{R}$ and $z\in \mathbb{R}$, define
$$f(x;z)=\int_{x}^\infty e^{-2iz(x-y)}w(y)dy=\int_{0}^\infty e^{2izy}w(x+y)dy.$$
Using the Plancherel's theorem, we have
$$
\|f(x;z) \|_{L^2_z}^2=\pi \int_{0}^\infty|w(x+y) |^2dy=\pi\int_{x}^\infty  |w(y)|^2dy.
$$
Further more, if $x\in \mathbb{R^{+}}$,we have
$$
\|f(x;z) \|_{L^2_z}^2=\pi \int_{0}^\infty|w(x+y) |^2dy=\pi\int_{x}^\infty  |w(y)|^2dy\leq\pi\langle x \rangle^{-2}\int_{x}^\infty  \langle y \rangle^{2}|w(y)|^2dy,
$$
which yields (\ref{8}).

Integrate by part, we get
\begin{equation}
2izf(x;z)+w(x)=\int_{-\infty}^xe^{-2iz(x-y)}\partial_yw(y)dy,
\label{48}\end{equation}
and
$$(2iz)^{2}f(x;z)+2izw(x)+\partial_xw(x)=\int_{x}^\infty e^{-2iz(x-y)}\partial^{2}_yw(y)dy.$$
We can easily get
$$
\left\|2iz\int_{x}^\infty e^{-2iz(x-y)}w(y)dy+w(x) \right\|_{L^2_z}^2={\pi}\int_{x}^\infty|\partial_yw(y)|^2dy,
$$
and
\begin{equation}
\begin{split}
&\left\|(2iz)^{2}\int_{x}^\infty e^{-2iz(x-y)}w dy+2izw(x)+\partial_xw(x) \right\|_{L^2_z }^2 = {\pi}\int_{x}^\infty
|\partial^{2}_yw(y)|^2dy.
\end{split} \nonumber
\end{equation}
Then, we get the equation (\ref{6}) for $n=1,2$. We can get $n=3$ by similar analysis.

Combine (\ref{8}) with (\ref{48}) we can get (\ref{45}).
Replacing $w$ with $(y-x)w$ and repeating above process, we can get (\ref{44}).
\end{proof}

\begin{proposition}\label{p1}
Let $u\in H^{3}(\mathbb{R})\cap H^{1,1}(\mathbb{R})$, $1-K_{u}$ is an invertible operator in  $L_{x}^{\infty}(\mathbb{R})\otimes L_{z}^{2}(\mathbb{R})$.
\end{proposition}
\begin{proof}
The conclusion is easy to get from Lemma \ref{lemma1}.
\end{proof}
\begin{proposition} \label{prop25}
Let $u\in H^{3}(\mathbb{R})$, then for every $x\in \mathbb{R}$, we have:
\begin{equation}
\begin{aligned}
&m_{1}^{\pm}-e_{1} \in L_{z}^{2}(\mathbb{R}),\\
&(2iz)(m_{1}^{\pm}-e_{1})-p_{1}^\pm \in L_{z}^{2}(\mathbb{R}),\\
&(2iz)^{2}(m_{1}^{\pm}-e_{1})-(2iz)p_{1}^\pm-q_{1}^\pm \in L_{z}^{2}(\mathbb{R}),\\
&(2iz)^{3}(m_{1}^{\pm}-e_{1})-(2iz)^{2}p_{1}^\pm-(2iz)q_{1}^\pm -g_{1}^\pm\in L_{z}^{2}(\mathbb{R}),
\end{aligned}
\label{2.52}\end{equation}
and
\begin{equation}
\begin{aligned}
&m_{2}^{\pm}-e_{2} \in L_{z}^{2}(\mathbb{R}),\\
&(2iz)(m_{2}^{\pm}-e_{2})-p_{2}^\pm \in L_{z}^{2}(\mathbb{R}),\\
&(2iz)^{2}(m_{2}^{\pm}-e_{2})-(2iz)p_{2}^\pm-q_{2}^\pm \in L_{z}^{2}(\mathbb{R}).
\end{aligned}
\label{2.54}\end{equation}
If $u\in H^{1,1}(\mathbb{R})$ then for every $x\in \mathbb{R}^{+}$, we have:
\begin{equation}
(2iz)\partial_{z}m_{1}^{\pm}\in L_{z}^{2}(\mathbb{R}).
\label{49}\end{equation}
 \end{proposition}

\begin{proof}
We only prove  (\ref{2.52}),(\ref{49}) while the others are similar to prove.

Recalling
$$
(1-K_{u})[m_{1}^{+}-e_{1}]=K_{u}e_{1},
$$
and from Proposition \ref{p1}, we know $1-K_{u}$ is an invertible operator in  $L_{x}^{\infty}(\mathbb{R})\otimes L_{z}^{2}(\mathbb{R})$.  Therefore  $m_{1}^{+}(x;\cdot)-e_{1}\in L^{2}(\mathbb{R})$  if  $K_{u}e_{1}(x;\cdot)\in L^{2}(\mathbb{R})$.  We write $K_{u}e_{1}$ in the following form:
$$
(K_{u}e_{1})(x;z)=[0,\sigma\int_{x}^{\infty}u(-y)e^{2iz(y-x)dy}]^{\top}.
$$
Recalling Lemma \ref{lemma3}, we know
$$
\sup _{x \in \mathbb{R}}\left\|\int_{x}^{\infty} \mathrm{e}^{2 i z(x-y)} \sigma u(-y) \mathrm{d} y\right\|_{L^{2}} \leq \sqrt{\pi}\|u\|_{L^{2}},
$$
then
\begin{equation}
\begin{aligned}
\|m_{1}^{+}-e_{1}\|_{L_{x}^{\infty}\otimes L_{z}^{2}}&\leq\|(1-K_{u})^{-1}
\|_{L_{x}^{\infty} \otimes L_{z}^{2} \to L_{x}^{\infty} \otimes L_{z}^{2} }\|K_{u}e_{1}\|_{L_{x}^{\infty} \otimes L_{z}^{2} }\\&
\leq c\|u\|_{L^{2}}.
\end{aligned}
\label{63}\end{equation}
Similar to above analysis, we can get
$$
(I-K_{u})[(2iz)(m_{1}^{+}-e_{1})-p_{1}^{+}]=(2iz)K_{u}e_{1}-(I-K_{u})p_{1}^{+},
$$
where
$$
(I-K_{u})p_{1}^{+}=[0,-u(-x)-\int_{x}^{\infty}e^{2iz(x_{1}-x_{2})}u(-x_{1})\int_{x_{1}}^{\infty}u(-x_{2})u(x_{2})dx_{2}dx_{1}]^{\top}.
$$
Then, we can get
\begin{equation}
\begin{aligned}
&(I-K_{u})[(2iz)(m_{1}^{+}-e_{1})-p_{1}^{+}]=2iz\int_{x}^{\infty}e^{-2iz(x-y)}\sigma u(-y)dy \\ &+\sigma u(-x)+\int_{x}^{\infty}e^{2iz(x_{1}-x)}u(-x_{1})\int_{x_{1}}^{\infty}u(x_{2})u(-x_{2})dx_{2}dx_{1}.
\end{aligned}
\label{53}\end{equation}
From Lemma \ref{lemma3} again, notice that $u\in H^2(\mathbb{R})$, then
\begin{equation}
\begin{aligned}
&\sup_{x\in\rm{R}}\left\| (2iz)\sigma\int_{x}^\infty e^{-2iz(x-y)}u(-y)dy+\sigma u(-x) \right\|_{L^2} \leq c\| \partial_x u\|_{L^2}
\end{aligned}\nonumber
\end{equation}
and
\begin{equation}
\begin{aligned}
&\sup_{x\in\rm{R}}\|\int_{x}^{\infty}e^{2iz(x_{1}-x)}u(-x_{1})\int_{x_{1}}^{\infty}u(-x_{2})u(x_{2})dx_{2}dx_{1}\|_{L^2} \leq c\|  u\|_{L^2}^{2}.
\end{aligned}\nonumber
\end{equation}
Where $c$ is a constant.

Combing with (\ref{53}) we have:
\begin{equation}
\begin{aligned}
&\sup _{x \in \rm{R}}\left\|\int_{x}^{\infty}(2iz) \mathrm{e}^{2 i z(x-y)} \sigma u(-y) \mathrm{d} y-(I-K_{u})p_{1}^{+}e_{2}\right\|_{L^{2}}\\
&\leq c(\|\partial_{x}u\|_{L^{2}}+\|u\|_{L^{2}}^{2}).
\end{aligned}\nonumber
\end{equation}
Inserting operator $I-K_{u}$ on $L^{2}_{x}(\mathbb{R})\otimes L^{2}_{z}(\mathbb{R})$, we get (\ref{2.52}) and (\ref{2.54}).

We derive on both side of  (\ref{2}),
$$
(1-K_{u})\partial_{z}m_{1}^{+}=(\partial_{z}K_{u})m_{1}^{+},
$$
then, we get
\begin{equation}
(1-K_{u})\partial_{z}m_{1}^{+}=(\partial_{z}K_{u})[m_{1}^{+}-e_{1}]+\partial_{z}K_{u}e_{1},
\label{61}\end{equation}
and
\begin{equation}
\begin{aligned}
&(I-K_{u})(2iz)\partial_{z}m_{1}^{+}\\
&=\partial_{z}K_{u}[(2iz)(m_{1}^{+}-e_{1})-p_{1}^{+}]+(2iz)\partial_{z}K_{u}e_{1}+\partial_{z}K_{u}p_{1}^{+}.
\end{aligned}
\label{70}\end{equation}

Recalling lemma \ref{lemma3}, we can get for every $x_{0}\in \mathbb{R^{+}}$,
\begin{equation}
\begin{aligned}
&\sup_{x\in(x_{0},\infty)}\|\langle x \rangle(m_{1}^{+}-e_{1})\|_{L^{2}}\\
&\leq c\|\int_{x}^\infty e^{-2iz(x-y)}\sigma u(-y)dy\|_{L^{2}}\leq c\|u\|_{L^{2,1}},
\end{aligned}
\label{65}\end{equation}
\begin{equation}
\begin{aligned}
&\sup_{x\in(x_{0},\infty)}\|\langle x \rangle(2iz)[(m_{1}^{+}-e_{1})-p_{1}^{+}]\|_{L^{2}}\\
&\leq c\|\int_{x}^\infty e^{-2iz(x-y)}\partial\sigma u(-y)dy\|_{L^{2}}\leq c\|u\|_{H^{1,1}}.
\end{aligned}
\label{64}\end{equation}
Combine (\ref{61}) with (\ref{65}) and recall lemma\ref{lemma3} we can get:
\begin{equation}
\begin{aligned}
&\|(I-K_{u})\partial_{z}m_{1}^{+}\|_{L^{2}}\leq\|\int_{x}^\infty 2i(y-x)\sigma u(-y)e^{-2iz(x-y)}(m_{11}^{+}-1)dy\|_{L^{2}}\\
&+\|\int_{x}^\infty 2i(y-x)e^{-2iz(x-y)}\sigma u(-y)dy\|_{L^{2}}\\
&\leq c\|u\|_{L^{1}}\sup_{x\in(x_{0},\infty)}\|\langle x \rangle(m_{1}^{+}-e_{1})\|_{L^{2}}+c\|u\|_{L^{2,1}}\leq c\|u\|_{H^{1,1}}.
\end{aligned}\nonumber
\end{equation}
Combine (\ref{64}) with (\ref{70}) and recall lemma\ref{lemma3} we can get:
\begin{equation}
\begin{aligned}
&\|(I-K_{u})(2iz)\partial_{z}m_{1}^{+}\|_{L^{2}}\leq\|\int_{x}^\infty 2i(y-x)e^{-2iz(x-y)}\sigma u(-y)[(2iz)(m_{11}^{+}-1)-p_{11}^{+}]dy\|_{L^{2}}\\
&+\|(2iz)\int_{x}^\infty 2i(y-x)e^{-2iz(x-y)}\sigma u(-y)dy\|_{L^{2}}+\|\int_{x}^\infty 2i(y-x)e^{-2iz(x-y)}\sigma u(-y)p_{11}^{+}dy\|_{L^{2}}\\
&\leq c\|u\|_{L^{1}}\sup_{x\in(x_{0},\infty)}\|\langle x \rangle[(2iz)(m_{1}^{+}-e_{1})-p_{1}^{+}]\|_{L^{2}}\leq c\|u\|_{H^{1,1}}.
\end{aligned}\nonumber
\end{equation}
Inserting operator $I-K_{u}$ on $L^{2}_{x}(\mathbb{R})\otimes L^{2}_{z}(\mathbb{R})$, we get (\ref{49}).
We finally improved the lemma.
\end{proof}

We has infect constructed the following maps:
\begin{equation}
\begin{aligned}
&H^{1,1}\ni u\rightarrow z\partial_{z}m_{1}^{\pm}(x;\cdot)\in  L^{2},\\
&H^{1,1}\ni u\rightarrow z\partial_{z}m_{2}^{\pm}(x;\cdot)\in  L^{2},\\
&H^{3}\ni u\rightarrow z^{3}(m_{1}^{\pm}(x;\cdot)-e_{1})-z^{2}p_{1}^\pm-zq_{1}^\pm -g_{1}^\pm\in  \L^{2},\\
&H^{3}\ni u\rightarrow z^{3}(m_{2}^{\pm}(x;\cdot)-e_{2})-z^{2}p_{2}^\pm-zq_{2}^\pm -g_{2}^\pm\in  \L^{2}.
\end{aligned}\nonumber
\end{equation}
 Next, we will show this map and remainders of the Jost function in function space $L^{\infty}_{x}(\mathbb{R})\otimes L^{2}_{z}(\mathbb{R})$ is Lipschitz continuous.
\begin{corollary}
Let $u,\tilde u\in  H^{1,1}(\mathbb{R})$ satisfy $\|u\|_{H^{1,1}(\mathbb{R})}\leq U$ and $\|\tilde{u}\|_{H^{1,1}(\mathbb{R })}\leq U$ for some $U>0$ then there is a positive U-dependent constant $C(U)$ such that for every $x\in \mathbb{R}$:
\begin{equation}
\left\|\check{m}_{1}^{\pm}(x ; \cdot)-\check{\tilde{m}}_{1}^{\pm}(x ; \cdot)\right\|_{L^2}+\left\|\check{m}_{2}^{\pm}(x ; \cdot)-\check{\tilde{m}}_{2}^{\pm}(x ; \cdot)\right\|_{L^2} \leq C(U)\|u-\tilde{u}\|_{H^{1,1}},\nonumber
\label{2.70}\end{equation}
where
\begin{align}
&\check{m}_{1}^{\pm}(x ; \cdot):=(2iz)\partial_{z}m_{1}^{\pm}(x;\cdot),\nonumber\\
&\check{m}_{2}^{\pm}(x ; \cdot):=(2iz)\partial_{z}m_{2}^{\pm}(x;\cdot).\nonumber
\end{align}
Moreover, if $u,\tilde u\in  H^{3}(\mathbb{R})$ satisfy $\|u\|_{ H^{3}(\mathbb{R})}\leq U$ and $\|\tilde{u}\|_{ H^{3}( \mathbb{R})}\leq U$ for some $U>0$ then there is a positive U-dependent constant $C(U)$ such that for every $x\in \mathbb{R}$:
$$
\left\|\hat{m}_{1}^{\pm}(x ; \cdot)-\hat{\tilde{m}}_{1}^{\pm}(x ; \cdot)\right\|_{L^2}+\left\|\hat{m}_{2}^{\pm}(x ; \cdot)-\hat{\tilde{m}}_{2}^{\pm}(x ; \cdot)\right\|_{L^2} \leq C(U)\|u-\tilde{u}\|_{H^{3}},\nonumber
$$
where
\begin{align}
&\hat{m}_{1}^{\pm}(x ; \cdot):=(2iz)^{3}(m_{1}^{\pm}(x;\cdot)-e_{1})-(2iz)^{2}p_{1}^\pm-(2iz)q_{1}^\pm -g_{1}^\pm,\nonumber\\
&\hat{m}_{2}^{\pm}(x ; \cdot):=(2iz)^{3}(m_{2}^{\pm}(x;\cdot)-e_{2})-(2iz)^{2}p_{2}^\pm-(2iz)q_{2}^\pm -g_{2}^\pm\nonumber.
\end{align}

\end{corollary}
\begin{proof}
Recalling (\ref{2}), we can get
\begin{align}
&m_{1}^{+}-\tilde{m}_{1}^{+}=(I-K_{u})^{-1}[K_{u}e_{1}-\tilde{K_{u}}e_{1}]+((I-K_{u})^{-1}-(I-\tilde{K_{u}})^{-1})\tilde{K_{u}}e_{1}\nonumber\\
&=(I-K_{u})^{-1}[K_{u}e_{1}-\tilde{K_{u}}e_{1}]+(I-K_{u})^{-1}(\tilde{K_{u}}-K_{u})(I-\tilde{K_{u}})^{-1}\tilde{K_{u}}e_{1}\nonumber.
\end{align}
Where
$$
\sup_{x\in\mathbb{R}}\|K_{u}e_{1}-\tilde{K_{u}}e_{1}\|_{L^{2}}=\sup_{x\in\mathbb{R}}\|\int_{x}^{\infty}e^{2iz(y-x)}(u-\tilde{u})dy\|_{L^{2}}.
$$
Using lemma\ref{lemma3} we can get
$$
\sup_{x\in\mathbb{R}}\|K_{u}e_{1}-\tilde{K_{u}}e_{1}\|_{L^{2}}\leq c\sup_{x\in\mathbb{R}}\|u-\tilde{u}\|_{L^{2}}.
$$
Furthermore, for every $f\in L^{2}_{z}(\mathbb{R})\otimes L^{\infty}_{x}(\mathbb{R})$,
\begin{equation}
\|(K_{u}-\tilde{K_{u}}) f\|_{L_x^{\infty}\otimes L_z^2} \leq C_2(U)e^{\|u-\tilde{u}\|_{L^{1}}}\|f\|_{L_x^{\infty}\otimes L_z^2}.\label{2.4}
\end{equation}
Then, we can get
$$
\sup_{x\in \mathbb{R}}\|m_{1}^{+}-\tilde{m}_{1}^{+}\|_{L^{2}}\leq c\|u-\tilde{u}\|_{L^{2}}.
$$
Using lemma\ref{lemma3} again, we can also get
\begin{equation}
\sup_{x\in(x_{0},\infty)}\|\langle x \rangle (m_{1}^{+}-\tilde{m}_{1}^{+})\|_{L^{2}}\leq c\|u-\tilde{u}\|_{L^{2,1}}.
\label{88}\end{equation}
Direct calculation yields:
\begin{equation}
\begin{aligned}
\ \partial_{z}m_{1}^{+}-\partial_{z}\tilde{m}_{1}^{+}  &=(I-K_{u})^{-1}\partial_{z}K_{u}(m_{1}^{+}-e_{1})-(I-\tilde{K_{u}})^{-1} \partial_{z}\tilde{K_{u}}(\tilde{m}_{1}^{+}-e_{1})\\
&-[(I-K_{u})^{-1}\partial_{z}K_{u}e_{1}-(I-\tilde{K_{u}})^{-1} \partial_{z}\tilde{K_{u}}e_{1}]\\
&=(I-K_{u})^{-1}[\partial_{z}K_{u}(m_{1}^{+}-e_{1})-\partial_{z}\tilde{K_{u}}(\tilde{m}_{1}^{+}-e_{1})]\\
&+(I-K_{u})^{-1}(\tilde{K_{u}}-K_{u})(I-\tilde{K_{u}})^{-1} \partial_{z}\tilde{K_{u}}\tilde{m}_{1}^{+}\\
&-(I-K_{u})^{-1}(\partial_{z}K_{u}e_{1}-\partial_{z}\tilde{K_{u}}e_{1})\\
&+(I-K_{u})^{-1}(\tilde{K_{u}}-K_{u})(I-\tilde{K_{u}})^{-1}  \partial_{z}\tilde{K_{u}}e_{1}.\label{2.2}
\end{aligned}
\end{equation}

Recalling lemma(\ref{lemma3}) and (\ref{88}), we get
\begin{equation}
\|(2iz)\partial_{z}K_{u}m_{1}^{+}-(2iz)\partial_{z}\tilde{K_{u}}\tilde{m}_{1}^{+}\|_{L^{2}}\leq C_{1}(U)\|u-\tilde{u}\|_{H^{1,1}},\label{2.3}
\end{equation}
Where $C_{1}(U)$ is an U-dependent positive constant. We notice that $1-K_{u}$ is an invertible operator in  $L_{x}^{\infty}(\mathbb{R})\otimes L_{z}^{2}(\mathbb{R})$ and recall (\ref{2.4}) we get the  bound (\ref{2.70}). The other  follow by repeating the same analysis.
\end{proof}

\subsection{Lipschitz continuity of scattering data  }

The propose of  Lemma \ref{lemma2} and Proposition \ref{prop25}  are to introduce the stand form of the scattering relations and study properties of that.
\begin{lemma}
If $u\in H^{1,1}(\mathbb{R}) $, then the function $a(z)$ is continued analytically in $\mathbb{C}^{+}$. In addition, we have:
$$
a(z)-1,d(z)-1,b(z) \in H^{1,1}(\mathbb{R}).
$$
Moreover, if $u\in H^{3}(\mathbb{R})$, then
$$
b(z) \in L^{2,3}(\mathbb{R}).
$$
\end{lemma}
\begin{proof}

By using (\ref{m}), (\ref{a})  and (\ref{b}),  we get
\begin{equation}
b(z) =\sigma\int_{\mathbb{R}}e^{2izy}u(-y)m_{11}^{+}dy,
\label{2.79}\end{equation}
\begin{equation}
\partial_{z}b(z)=\sigma\int_{\mathbb{R}}e^{2izy}u(-y)\partial_{z}m_{11}^{+}dy+\sigma\int_{\mathbb{R}}2iye^{2izy}u(-y)m_{11}^{+}dy,
\label{78}\end{equation}
\begin{equation}
\begin{aligned}
&a(z)-1 =-\int_{\mathbb{R}}u(y)m_{21}^{+}dy,\\
&d(z)-1 =\sigma\int_{\mathbb{R}}u(-y)m_{12}^{+}dy,
\end{aligned}
\label{2.78}\end{equation}
\begin{equation}
\begin{aligned}
&\partial_{z}a(z)=-\int_{\mathbb{R}}u(y)\partial_{z}m_{21}^{+}dy,\\
&\partial_{z}d(z)=\sigma\int_{\mathbb{R}}u(-y)\partial_{z}m_{21}^{+}dy.
\end{aligned}
\label{76}\end{equation}

We can easily obtain the following limit for the scattering coefficient $a(z)$ along a contour in $\mathbb{C}^{+}$ extended to $Imz\rightarrow \infty$:
$$
a(z)\rightarrow 1
$$
In order to prove $b(z)\in L_{z}^{2,3}(\mathbb{R})$, we rewrite (\ref{b}) as following form:
\begin{equation}
\begin{aligned}
&(2iz)^{3}b(z)=\sigma\int_{\mathbb{R}}e^{2izy}u[(2iz)^{3}(m_{11}^{+}-1)-(2iz)^{2}p_{11}^{+}-(2iz)q_{11}^{+}-g_{11}^{+}]dy\\
&+(2iz)^{3}\sigma\int_{\mathbb{R}}e^{2izy}udy+(2iz)^{2}\sigma\int_{\mathbb{R}}e^{2izy}up_{11}^{+}dy+(2iz)\sigma\int_{\mathbb{R}}e^{2izy}uq_{11}^{+}dy\\
&+\sigma\int_{\mathbb{R}}e^{2izy}ug_{11}^{+}dy.
\end{aligned}
\label{80}\end{equation}
Using lemma\ref{lemma3}, we get

\begin{align}
&\|\sigma\int_{\mathbb{R}}e^{2izy}u[(2iz)^{3}(m_{11}^{+}-1)-(2iz)^{2}p_{11}^{+}-(2iz)q_{11}^{+}-g_{11}^{+}]dy\|_{L^{2}}\nonumber\\
&\leq c\|u\|_{L^{1}}\sup_{x\in\mathbb{R}}\|(2iz)^{3}(m_{11}^{+}-1)-(2iz)^{2}p_{11}^{+}-(2iz)q_{11}^{+}-g_{11}^{+}\|_{L^{2}}\nonumber\\
&\leq c\|u\|_{L^{1}}\|u\|_{H^{3}}\nonumber.
\end{align}

We proved the first term of (\ref{80}) is bounded in $L^{2}$ space. For the other terms of (\ref{80}), using  Plancherel's formula, we can get:
\begin{align}
&\|(2iz)^{3}\int_{\mathbb{R}}e^{2izy}udy\|_{L^{2}}\leq c\|u\|_{H^{3}}\nonumber,\\
&\|(2iz)^{2}\int_{\mathbb{R}}e^{2izy}up_{11}^{+}dy\|_{L^{2}}\leq c\|u\|_{H^{3}}\nonumber,\\
&\|(2iz)\int_{\mathbb{R}}e^{2izy}uq_{11}^{+}dy\|_{L^{2}}\|_{L^{2}}\leq c\|u\|_{H^{3}}\nonumber,\\
&\|\int_{\mathbb{R}}e^{2izy}ug_{11}^{+}dy\|_{L^{2}}\|_{L^{2}}\leq c\|u\|_{H^{3}}\nonumber.
\end{align}
Then we get $b(z) \in L^{2,3}(\mathbb{R})$.

In order to prove $\partial_{z}b(z)\in L_{z}^{2,1}(\mathbb{R})$, we rewrite (\ref{78}) as following form:
\begin{equation}
\begin{aligned}
(2iz)\partial_{z}b(z)&=\sigma\int_{\mathbb{R}}e^{2izy}u(2iz)\partial_{z}m_{11}^{+}+\sigma\int_{\mathbb{R}}(2iy)e^{2izy}u[(2iz)(m_{11}^{+}-1)-p_{11}^{+}]dy\\
&+(2iz)\sigma\int_{\mathbb{R}}(2iy)e^{2izy}udy+\sigma\int_{\mathbb{R}}(2iy)e^{2izy}up_{11}^{+}dy.
\end{aligned}
\label{81}\end{equation}
Using Lemma  \ref{lemma3}  again, we get
\begin{align}
&\|\int_{\mathbb{R}}e^{2izy}u(2iz)\partial_{z}m_{11}^{+}dy\|_{L^{2}}\leq c\|u\|_{L^{1}}\sup_{x\in\mathbb{R}}\|(2iz)\partial_{z}m_{11}^{+}\|_{L^{2}}\nonumber\\
&\leq c\|u\|_{L^{1}}\|u\|_{H^{1,1}}\nonumber.
\end{align}
We proved the first term of (\ref{81}) is bounded in $L^{2}$ space. For the other terms of (\ref{81}), using  Plancherel's formula, we can get:

\begin{align}
&\|(2iz)\int_{\mathbb{R}}e^{2izy}(2iy)udy\|_{L^{2}}\leq c\|u\|_{H^{1,1}}\nonumber,\\
&\|\int_{\mathbb{R}}e^{2izy}(2iy)up_{11}^{+}dy\|_{L^{2}}\leq c\|u\|_{H^{1,1}}\nonumber.
\end{align}
We get $\partial_{z}b(z) \in L^{2,1}(\mathbb{R})$.
Then we have proved the bounds of $b(z)$. The conclusion of $a(z),d(z)$ can be proved similar.
\end{proof}
It is easy to find that we have constructed the following two maps:
$$
H^{1,1}(\mathbb{R})\ni u\rightarrow a(z)-1,d(z)-1, b(z)\in H^{1,1}(\mathbb{R})
$$
and
$$
 H^{3}(\mathbb{R})\ni u\rightarrow  b(z)\in L^{2,3}(\mathbb{R}).
$$
We will show the two maps are Lipschitz continuous.
\begin{corollary}
Let $u\in  H^{1,1}(\mathbb{R})\cap H^{3}(\mathbb{R})$ then,
$$
H^{1,1}(\mathbb{R})\ni u\rightarrow a(z)-1,d(z)-1, b(z)\in H_{z}^{1,1}(\mathbb{R})
$$
and
$$
H^{3}(\mathbb{R})\ni u\rightarrow  b(z)\in L_{z}^{2,3}(\mathbb{R})
$$
are Lipschitz continuous.
\end{corollary}
\begin{proof}
From the representations (\ref{2.79}), (\ref{78}), (\ref{2.78}), (\ref{76}) and the Lipschitz continuity of  the Jost function $m_{1}^{\pm}$ and $m_{2}^{\pm}$, we can obtain the Lipschitz continuity of the scattering coefficients.
\end{proof}

For the reflection coefficients $r(z)$, we have the following results:
\begin{lemma}
If $ u\in  H^{1,1}(\mathbb{R})\cap H^{3}(\mathbb{R})$  then we have $$r_{1,2}(z)\in H^{1,1}(\mathbb{R})\cap L^{2,3}(\mathbb{R}).$$
 As well, the mapping
$$
 H^{1,1}(\mathbb{R})\cap H^{3}(\mathbb{R})\ni u \mapsto r_{1,2}\in H^{1,1}(\mathbb{R})\cap L^{2,3}(\mathbb{R})
$$
is Lipschitz continuous.
\end{lemma}
\begin{proof}

Let $r_{1}$ and $\tilde r_{1}$ denote the reflection coefficients corresponding to $u$ and $\tilde u$, respectively.
Owing to
$$
r_{1}-\tilde r_{1}=\frac{b-\tilde b}{a}+\frac{\tilde b((\tilde a-1)-(a-1))}{a\tilde a},
$$
the Lipschitz continuity of $r_{1}$ follows from the Lipschitz continuity of $a-1$ and $b$ and we can get similar conclusion of $r_{2}$ .
\end{proof}
\begin{lemma}
 If $u\in  H^{1,1}\cap H^{3} $ with $L^1$-norm such that  $  \| u\|_{L^1 } \exp \{2\| u\|_{L^1 }\}<1$,
 then the spectral problem  admits no eigenvalues or resonances, the scattering coefficients $a(z)$ and $d(  z)$ admit no zeros in $\mathbb{C}^+\cup \mathbb{R}$ and $\mathbb{C}^-\cup \mathbb{R}$, respectively.
\end{lemma}
\begin{proof}
The small-norm condition implies that $\| u\|_{L^1(\mathbb{R})}<1$.
Recall that $m_{1}^{+}=e_1+K_{u}m_{1}^{+}$ in Lemma \ref{lemma1}  and the operator $I-K_{u}$ is invertible and bounded from $ L^{\infty}_{x} $ to $ L^{\infty}_{x} $ and we reach that for every $z\in \mathbb{C}^+$,
$$
\|m_{1}^{+}(\cdot;z)-e_1\|_{L^{\infty}}=\|(I-K_{u})^{-1}\|\| K_{u}e_1(\cdot;z)\|_{L^{\infty}}\leq \| u\|_{L^1 } e^{2\| u\|_{L^1}}.
$$
 We derive for every $z\in \mathbb{C}^+$,
$$
|a(z)|\geq 1-\left|       \int_{\mathbb{R}} u(y)m_{21}^{+}(y;z)dy      \right|> 1- \| u\|_{L^1} e^{2\| u\|_{L^1}}.\nonumber
$$
Due to the continuity of $a(z)$, we also obtain that
$$
|a(z)|\geq  1-  \| u\|_{L^1(\mathbb{R})} e^{2\| u\|_{L^1 }}>0, \ \ z\in \mathbb{R},
$$
then  $a(z)$ admits no zeros in  $\mathbb{C}^+\cup \mathbb{R}$.
Carrying out a similar manipulation for $d(z)$, we see that $d(z)$ admits no zeros in  $\mathbb{C}^-\cup \mathbb{R}$.
\end{proof}
\begin{lemma}
If $u\in L^{2,1}(\mathbb{R})$  with $L^1(\mathbb{R})$ small-norm such that
\begin{align}
 1-\| u\|_{L^1(\mathbb{R})}(1+2 e^{2\| u\|_{L^1(\mathbb{R})} }) >0, \label{cond2}
 \end{align}
then for every $z\in\mathbb{R}$, we have $|r_{1,2}(z)|<1$.
\end{lemma}
\begin{proof}
Rewrite  (\ref{2.79}) for $b(z)$ as
\begin{equation}
b(z) =\sigma\int_{\mathbb{R}}e^{2izy}u(-y)m_{11}^{+}dy,
\end{equation}
Under the condition (\ref{cond2}).For every $z\in\mathbb{R}$,  we obtain
\begin{align}
&|b(z)|\leq \|m_{11}^{+}(\cdot;z)-1 \|_{L^{\infty}}\| u\|_{L^1}+\|u \|_{L^1}\leq \| u\|_{L^1(\mathbb{R})}  +  \| u\|_{L^1(\mathbb{R})}e^{2\| u\|_{L^1(\mathbb{R})} }\nonumber\\
&<  1-\| u\|_{L^1(\mathbb{R})}  e^{2\| u\|_{L^1(\mathbb{R})} } \leq  a(z),\nonumber
\end{align}
which yields
$$|r_1(z)|=\frac{|b(z)|}{|a(z)|} <1.$$
Similarly, we get $|r_2(z)|< 1$.
\end{proof}

\section{Inverse scattering transform} \label{sec3}

In this section, we will set up a RH problem
 and show  the existence and uniqueness of its  solution    for the given data $r(z)\in H^{1,1}(\mathbb{R})\cap L^{2,3}(\mathbb{R})$ .

\subsection{Set-up of RH problem}

Define
$$
M(x, z)=\begin{cases}
\begin{pmatrix}
m_{1}^{(+)}&m_{2}^{(-)}
\end{pmatrix}\begin{pmatrix}
a^{-1} & 0 \\
0 & 1
\end{pmatrix}& \operatorname{Im} z>0\\
\begin{pmatrix}
m_{1}^{(-)}& m_{2}^{(+)}
\end{pmatrix}
\begin{pmatrix}
1 & 0 \\
0 & d^{-1}
\end{pmatrix},&\operatorname{Im} z<0.
\end{cases}
$$
 then  $M (x, z) $ satisfy the following    RH-problem

\begin{problem} \label{ope}
 Find a   matrix function  $M(x ; z)$ satisfying

(i)  $M(x ; z) \rightarrow I+\mathcal{O}\left(z^{-1}\right)$ as $z \rightarrow \infty$.

(ii)   For   $  M(x ; z)$ admits the following jump condition
\begin{equation}
\begin{aligned}
M_{+}(x; z) &   =M_{-}(x; z) V_{x}(z), \label{prrwe}
\end{aligned}
\end{equation}
where
\begin{equation}
V_{x}(z):=\left(\begin{array}{cc}
1+\sigma r_{1}r_{2} & \sigma r_{2}(z) e^{2 \mathrm{i} x z}\\
r_{1}(z)e^{-2 \mathrm{i} x z} & 1
\end{array}\right),z\in \mathbb{R}.
\end{equation}
The reconstruction formula is given by
\begin{equation}
u(x):=2 \mathrm{i} \lim _{z \rightarrow \infty} z M_{12}(x, z),
u(-x):=2 \mathrm{i}\sigma \lim _{z \rightarrow \infty} z M_{21}(x, z).
\label{3.9}\end{equation}
\label{RHP1}\end{problem}

  We write (\ref{prrwe}) in the form
$$
    M_+(x;z)-M_-(x;z)=M_-(x;z)S(x;z),\quad z\in \mathbb{R},
$$
where
$$
S(x;z)=\left(\begin{array}{cc}
\sigma r_{1}r_{2} & \sigma r_{2}(z) e^{2 \mathrm{i} x z}\\
r_{1}(z)e^{-2 \mathrm{i} x z} & 0
\end{array}\right),
$$
Introduce a transformation
$$\Psi_{\pm}(x;z)=M_{\pm}(x;z)-{I},$$
then we obtain a new RH problem for $\Psi(x;z)$
\begin{align}
    &\Psi_+(x;z)-\Psi_-(x;z)=\Psi_-(x;z)S(x;z)+S(x;z),\quad z\in \mathbb{R},\nonumber\\
    &\Psi_{\pm}(x;z)\rightarrow 0,\quad |z|\rightarrow \infty,\ z\in \mathbb{C}\backslash \mathbb{R}\nonumber.
\end{align}
\subsection{Solvability   of  the RH problem }

We introduce  the Cauchy operator
$$
 \mathcal{C} ( f) (z)=\frac{1}{2 \pi \mathrm{i}} \int_{\mathbb{R}} \frac{f(\zeta)}{\zeta-z} d \zeta, \quad z \in \mathbb{C} \backslash \mathbb{R}
$$
and Plemelj projection operator
\begin{equation}
 \mathcal{P}^{\pm} (f)(x;z):=  \lim _{\varepsilon \searrow 0} \frac{1}{2 \pi \mathrm{i}} \int_{\mathbb{R}} \frac{f(\zeta)}{\zeta-(z \pm \mathrm{i} \varepsilon)} d \zeta, \quad z \in \mathbb{R},
\label{3.12}\end{equation}
where $f(z) \in L^{2}(\mathbb{R})$.
\begin{proposition}
(  \cite{Pelinovsky}) For every $f\in L^p(\mathbb{R})$ with $\ 1\leq p<\infty $, the Cauchy operator $\mathcal{C}(f)$ is analytic off the real line, decays to zero as $|z|\rightarrow\infty$, and approaches to $\mathcal P^{\pm}(f)$ almost everywhere when a point $z\in C^{\pm}$ approaches to a point on the real axis by any non-tangential contour from $\mathbb{C}^{\pm}$. If $1< p<\infty$, then there exists a positive constant $C_p$( $C_{p }=1, p=2$  ) such that
 \begin{equation}
    \|   \mathcal{P}^{\pm}(f)\|_{L^p}\leq C_p \| f\|_{L^p},
\label{3.13}\end{equation}
If $f\in L^1(\mathbb{R})$, then the Cauchy operator admits  the  asymptotic
\begin{equation}
    \lim_{|z|\rightarrow \infty} z\mathcal C(f)(z)=-\frac{1}{2\pi i}\int_\mathbb{R} f(s)d s.
\label{3.14}\end{equation}
\end{proposition}
\begin{lemma}
 Let $r_{1,2}(z)\in H^1(\mathbb{R})$ , satisfy $|r_{1,2}(z)|\leq1$ then there exist positive constants $c_-$ and $c_+$ such that for every $x\in \mathbb{R}$ and every column-vector $g\in \mathbb{C}^2$, we have
\begin{equation}
    {\rm Re}\ g^*( I+S(x;z))g\geq c_-g^*g,\quad z\in \mathbb{R}
\label{3.15}\end{equation}
and
\begin{equation}
    \|(I+S(x;z))g\|\leq c_+\|g \|,\quad z\in \mathbb{R},
\label{3.16}\end{equation}
where the asterisk denotes the Hermite conjugate.
\end{lemma}
\begin{proof}
The original scattering matrix $S(x;z)$ is not Hermitian due to the fact that  there is no relationship between $a(z)$ and $d(z)$.
Therefore, we define Hermitian part of $S(x;z)$ by
\begin{equation}
\begin{split}
    S_H(x;z)&=\frac 12( S(x;z)+  S^*(x;z))\\
    &=\left( \begin{array}{cc}
    \sigma \rm Re(r_1r_2) & \frac12(\bar r_1+\sigma r_2)e^{2izx}\\
    \frac 12(r_1+\sigma \bar{r}_2)e^{-2izx} & 0
    \end{array}\right).
\end{split}
\end{equation}
Since $|r_{1,2}(z)|<1$, the  2-order principle minor of the matrix $I+S_H$
$$1+\sigma{\rm Re}(r_1r_2)-\frac14|r_1+\sigma\bar r_2|^2=1-\frac14|r_1-\sigma\bar r_2|^2>0,$$
 which indicates that the 1-order  principle minor $1+\sigma\rm{Re}(r_1r_2)>0$. Thus the matrix $I+S_H$ is positive definite.

In view of the algebra theory, for a Hermitian matrix, there exists an unitary matrix $A$  such that
\begin{equation}
A^*(I+S_H)A={\rm diag}(
\mu_+, \mu_- ),\label{unitary}
\end{equation}
where $\mu_{\pm}$ are the eigenvalues of the matrix $I+S_H$
\begin{equation}
    \mu_{\pm}(z)=\frac{2+\sigma \rm{Re}(r_1r_2)\pm\sqrt{\rm{Re}^2(r_1r_2)+|r_1+\sigma \bar{r}_2|^2}}{2}.
    \nonumber
\end{equation}
    Note $\mu_+(z)>\mu_-(z)>0$ as $I+S_H$ is positive definite. And it follows from $r_{1,2}(z)\rightarrow 0$  that $\mu_-(z)\rightarrow 1$ as $|z|\rightarrow\infty,\ z\in\mathbb{R}$. Together with  $r_{1,2}(k)\in H^1(\mathbb{R})$, there exists a positive constant $c_-$ such that $\mu_->c_-$.

Consequently, for every $g\in \mathbb{C}^2$, utilizing (\ref{unitary}),
  we have
\begin{equation*}
    c_-g^*g<\mu_-g^*g\leq{\rm Re}  g^*( I+S(x;z))g=g^*(I+S_H)g,
\end{equation*}
this completes the proof of the bound (\ref{3.15}).

Calculating $(I+S(x;z))g$ componentwise and utilizing $|r_{1,2}(z)|<1$ gives that
\begin{equation}
    \begin{split}
       \|(I+S(x;z))g \|^2\leq & 2(1+|r_1 |+|r_2 |)^2\| g \|^2\\
       &+2 {\rm Re}\{ ((\sigma+r_1r_2)\bar r_2+r_1)e^{-2izx}g^{(1)}\overline{g^{(2)}}\}\\
       \leq&\left((|r_1 |+1)^2+(|r_2 |+1)^2+(|r_1 |+|r_2 |)^2\right)\| g \|^2,
    \end{split}
\end{equation}
here the norm for a 2-component vector $f$ is $\|f\|^2=| f^{(1)}|^2+| f^{(2)}|^2$.
Therefore, we take
$$c_+=\sup_{z\in\mathbb{R}}\sqrt{(|r_1 |+1)^2+(|r_2 |+1)^2+(|r_1 |+|r_2 |)^2}<+\infty,$$
then one obtain the bound (\ref{3.16}).
\end{proof}

\begin{lemma}
Let $r_{1,2}(z)\in H^1(\mathbb{R})$ ,satisfy $|r_{1,2}(z)|\leq1$ then for every $F(z)\in L^2_z(\mathbb{R})$, there exists a unique solution $\Psi(z)\in L_z^2(\mathbb{R})$ of the equation
\begin{equation}
(  I-P_S^-)\Psi(z)=F(z),\ z\in \mathbb{R}, \label{eie}
\end{equation}
where $P_S^-\Psi =P^-(\Psi S)$.
\end{lemma}
\begin{proof}
Since $  I-P_S^-$ is a Fredholm operator of the index zero, by Fredholm's alternative theorem, there exists a unique solution
of the equation (\ref{eie})  if and only if   the homogeneous  equation
\begin{equation}
(  I-P_S^-)g=0,\ z\in \mathbb{R}, \label{eie2}
\end{equation}
admits zero solution  in $L^2_z(\mathbb{R})$.

Assume that $g(z)\in L^2_z(\mathbb{R})$ and $g(z)\neq 0$  is a solution of equation (\ref{eie2}).
Define two analytic functions in $\mathbb C\backslash  \mathbb{R}$
$$
    g_1(z)=\mathcal C(gS)(z), \quad g_2(z)=\mathcal C(gS)^*(z).
$$
 The functions $g_1(z)$ and $g_2(z)$ are well-defined due to $S(z)\in L_z^2(\mathbb R)\cap L_z^{\infty}(\mathbb R)$.

 We integrate the function $g_1(z)g_2(z)$ along the semi-circle of radius R centered at zero in $\mathbb C^+$. It follows from Cauchy theorem that
$$
    \oint g_1(z)g_2(z)dz=0.
$$
Since $g(z)S(z)\in L_z^1(\mathbb R)$,  using  yields $g_{1}(z), g_{2}(z)=\mathcal O(z^{-1}),\ |z|\rightarrow \infty$. Hence, the integral on the arc approaches to zero as the radius approaches to infinity. Therefore, we obtain
\begin{equation}
    \begin{split}
        0&=\int_{\mathbb{R}} g_1(z)g_2(z) dz =\int_{R} \mathcal{P}^+(gS)[\mathcal{P}^-(gS)]^*\rm dz\\
        &=\int_{\mathbb{R}}[P^-(gS)+gS][\mathcal{P}^-(gS)]^*\rm dz.
    \end{split}
\end{equation}
Utilizing the assumption $\mathcal{P}^-(gS)=g$, we have
$$
\int_{R} g(I+S)g^*dz=0.
$$
 We get ${\rm Re}\ g(I+S)g^*>c_-g^*g$ with $c_-$ is a positive constant. Thus the function $g(z)$ has to be  zero. This contradicts to the assumption $g\neq 0$. Therefore, $g=0$ is a unique solution to the equation $(I-P^-_S)g=0$ in $L^2_z(\mathbb{R})$. Finally there exists a unique solution to the equation $(I-P_S^-)\Psi(z)=F(z)$.
\end{proof}

\begin{lemma}
 Let $r_{1,2}(z)\in H^1(\mathbb{R})$ satisfy $|r_{1,2}(z)|\leq1$ then for every $x\in \mathbb{R}$,
there exist  unique solutions $\Psi_{\pm}(x;z)\in L^2_z(\mathbb{R})$ satisfying the equation
$$
\Psi_+(x;z)-\Psi_-(x;z)=\Psi_-S(x;z)+S(x;z),\quad z\in \mathbb{R}.
$$

\label{lemma3.4}\end{lemma}

\begin{proof}
 Owing to $S(x;z)\in L^2_z(\mathbb{R})$, we have $\mathcal{P_S}^-(z)\in L^2_z(\mathbb{R})$ by (\ref{3.13}). Then  for every $x\in \mathbb{R}$,
 there exists a unique solution $\Psi_-(x;z)\in L^2_z(\mathbb{R})$  satisfying the   equation
 \begin{equation}
      \Psi_-(x;z)=\mathcal{P}^-(\Psi_-(x;z)S(x;z)+S(x;z)),\quad z\in \mathbb{R}.
 \label{3.23}\end{equation}
 Based on the existence of $\Psi_-(x;z)$, we define a function $\Psi_+(x;z)$ by
 \begin{equation}
     \Psi_+(x;z)=\mathcal{P}^+(\Psi_-(x;z)S(x;z)+S(x;z)),\quad z\in \mathbb{R}.
 \label{3.24}\end{equation}
  Besides, analytic extensions of $\Psi_{\pm}(x;z)$ to $z\in \mathbb{C}^{\pm}$ are defined by Cauchy operator
 \begin{equation}
     \Psi_{\pm}(x;z)=\mathcal C(\Psi_-(x;z)S(x;z)+S(x;z)),\quad z\in \mathbb{C}^{\pm}.
 \label{3.25}\end{equation}
 Finally we obtain the solution $\Psi_{\pm}(x;z)\in L^2_z(\mathbb{R})$. Moreover, given the property of the Cauchy operator and the Plemelj projection operator, the solutions $\Psi_{\pm}(x;z)$ are analytic functions for $z\in \mathbb{C}^{\pm}$.
\end{proof}
\begin{lemma}
 Let $r_{1,2} \in H^1(\mathbb{R})$ ,satisfy $|r_{1,2}(z)|\leq1$  then the operator $(I-\mathcal{P}_S^-)^{-1}$ is bounded from $L_z^2(\mathbb{R})$ to $L_z^2(\mathbb{R})$, and there exists a constant $c$ that only depends on $\| r(z)\|_{L_z^{\infty}}$ such that
$$
\|(I-P_S^-)^{-1} f \|_{L_z^2}\leq c\| f\|_{L_z^2}
$$
\end{lemma}
\begin{proof}
For every $f(z)\in L_z^2( \mathbb{R})$  that there exists a solution $\Psi(z)\in L_z^2( \mathbb{R})$ to $(I-\mathcal{P}_S^-)\Psi(z)=f(z)$. Note that $\mathcal{P}^+-\mathcal{P}^{-}=I$, we decompose the function into $\Psi=\Psi_+-\Psi_-$ with
\begin{equation}
\Psi_--\mathcal{P}^-(\Psi_-S)=\mathcal{P}^-(f),\quad \Psi_+-\mathcal{P}^-(\Psi_+S)=\mathcal{P}^+(f).
\label{3.27}\end{equation}
Since $\mathcal{P}^{\pm}(f)\in L_z^2(\mathbb{R})$,   there exist unique solutions $\Psi_{\pm}(z)\in L_z^2(\mathbb{R})$ for (\ref{3.23}) which implies the decomposition is unique.
Therefore, we only need the estimates of $\Psi_{\pm}$ in $L^2_z(\mathbb{R})$.

To deal with $\Psi_-$, define two analytic functions in $\mathbb{C}\backslash \mathbb{R}$
$$g_1(z)=\mathcal C(\Psi_-S)(z),\quad g_2(z)=\mathcal C(\Psi_-S+f)^*(z),$$
Analogous manipulation , we integrate on the semi-circle in the upper half-plane and have
$$
    \oint g_1(z)g_2(z)dz=0.
$$
Since $g_1(z)=\mathcal O(z^{-1})$ and $g_2(z)\rightarrow 0$ as $|z|\rightarrow\infty$, we have
\begin{equation}
    \begin{split}
        0=&\int_{\mathbb{R}}\mathcal{P}^+(\Psi_-S)[\mathcal{P}^-(\Psi_-S+f)]^* dz\\
        =&\int_{\mathbb{R}}(\mathcal{P}^-(\Psi_-S)+\Psi_- S)[\mathcal{P}^-(\Psi_-S+f)]^* dz\\
        =&\int_{\mathbb{R}}(\Psi_--\mathcal{P}^-(f)+\Psi_-S)\Psi_-^* dz.
    \end{split}
\end{equation}
 Using the bound (\ref{3.15}) , (\ref{3.16}) and the H\"{o}lder inequality, there exists a positive constant $c_-$ such that
$$
    c_-\| \Psi_-\|_{L^2}^2\leq {\rm Re} \int_{\mathbb{R}} \Psi_-(I+S)\Psi_-^* dz={\rm  Re}\int_{\mathbb{R}} \mathcal{P}^-(f)\Psi_-^*dz\leq \| f\|_{L^2}\| \Psi_-\|_{L^2},\nonumber
$$
which  completes the estimates of $\Psi_-$:
\begin{equation}
 \|(I-\mathcal{P}_S^{-})^{-1}P^-f \|_{L^2_z}\leq c_-^{-1}\| f\|_{L^2_z}.\label{f1}
\end{equation}
To deal with $\Psi_+$, define two functions in $\mathbb{C}\backslash \mathbb{R}$
$$g_1(z)=\mathcal C(\Psi_+S)(z),\quad g_2(z)=\mathcal C(\Psi_+S+f)^*(z).$$
Performing the similar procedure leads to
\begin{equation}
    \begin{split}
        0=&\oint g_1(z)g_2(z) dz\\
        =&\int_{\mathbb{R}} \mathcal{P}^-(\Psi_+S)[\mathcal{P}^+(\Psi_+S+f)]^* dz\\
        =&\int_{\mathbb{R}}[\Psi_+-\mathcal{P}^+(f)][\Psi_+(I+S)]^* dz,
        \end{split}
\end{equation}
where we have used (\ref{3.23}).
Using the bounds (\ref{3.15}) and (\ref{3.16}) , there are positive constants $c_-$ and $c_+$ such that
\begin{equation}
    c_-\| \Psi_+\|_{L^2}^2\leq {\rm Re} \int_{\mathbb{R}} \Psi_+(I+S)^*\Psi_+^* dz={\rm Re} \int_{\mathbb{R}}\mathcal{P}^+(f)(I+S)^*\Psi_+^* dz\leq c_+\|f \|_{L^2}\| \Psi_+\|_{L^2},\nonumber
    \end{equation}
which means
\begin{equation}
    \| (I-P_S^-)^{-1}P^+f\|_{L^2_z}\leq c_-^{-1}c_+\| f\|_{L^2_z}.\label{f2}
\end{equation}
Combining (\ref{f1}) and (\ref{f2}), we obtain
$$\|(I-\mathcal{P}_S^-)^{-1} f \|_{L_z^2}\leq c\| f\|_{L_z^2},$$
where $c$ is a constant that only depends on $\| r(z)\|_{L_z^{\infty}}$.
\end{proof}
\subsection{Estimate on solutions to  the RH problem}
Next, we see solutions to  the RH problem for $M(x;z)$. Denote the functions $M_{\pm}$ column-wise
$$M_{\pm}(x;z)=[\mu_{\pm}(x;z), \nu_{\pm}(x;z)], $$
then  the functions $\Psi_{\pm}$ can be written  as
$$\Psi_{\pm}(x;z)=[\mu_{\pm}(x;z)-e_1, \nu_{\pm}(x;z)-e_2 ].$$
We have
\begin{equation}
    \mu_{\pm}(x;z)-e_1 =\mathcal{P}^{\pm}(M_-S)^{(1)}(x;z),\quad z\in{\mathbb{R}}\label{3.36}
\end{equation}
and
\begin{equation}
    \nu_{\pm}(x;z)-e_2 =\mathcal{P}^{\pm}(M_-S)^{(2)}(x;z),\quad z\in{\mathbb{R}}.\label{3.37}
\end{equation}
Combining (\ref{3.36}) with (\ref{3.37}), we obtain
\begin{equation}
M_{\pm}(x;z)={I}+\mathcal{P}^{\pm}(M_-(x;\cdot)S(x;\cdot))(z),\quad z\in \mathbb{R}.
\label{pleme}\end{equation}
There exist  unique solutions $M_{\pm}(x;z)$ to Eq.(\ref{pleme}).
Further, analytic extensions of $M_{\pm}(x;z)$ to $z\in \mathbb{C}^{\pm}$ are
\begin{equation}
M_{\pm}(x;z)={I}+\mathcal C (M_-(x;\cdot)S(x;\cdot))(z),\quad z\in{\mathbb{C}}^{\pm}.
\label{79}\end{equation}
\begin{lemma}\label{3.4} Let $r_{1,2}(z)\in H^1(\mathbb{R})$, satisfy $|r_{1,2}(z)|\leq1$ then there exists a constant $c$ only depending on $\| r(z)\|_{L^{\infty}}$ such that for every $x\in \mathbb{R}$,
\begin{equation}
    \| M_{\pm}(x;\cdot)-I\|_{L^2}\leq c(\|r_{1} \|_{L^2}+\|r_{1} \|_{L^2}).\label{3.42}
\end{equation}
\end{lemma}
\begin{proof}
Due to $r_{1,2} \in H^1 (\mathbb{R})$, we get  $r_{1,2} \in L^2(\mathbb{R}) \cap L^{\infty}(\mathbb{R}) $ and $S(x;\cdot)\in L^2_z $.
Moreover, there exists a constant $c$ only depending on $\|r_{1,2} \|_{L^{\infty}}$ such that
\begin{equation}
    \|S(x;\cdot) \|_{L_z^2}\leq c(\|r_{1} \|_{L^2}+\|r_{1} \|_{L^2}),\nonumber
\end{equation}
We obtain
\begin{equation}
    \|M_{\pm}-I \|_{L^{2}}=\| \psi_{\pm}\|_{L^2}
    \leq c(\|r_{1} \|_{L^2}+\|r_{1} \|_{L^2}),\nonumber
\end{equation}
where we have used the equation $(I-\mathcal{P}_S^-)\Psi_{\pm}=\mathcal{P}^-S$, and $c$ is another constant only depending on $\| r_{1,2}(z)\|_{L^{\infty}}$.
\end{proof}
\begin{proposition} For every $x_0\in \mathbb{R}^-$ and every $r_{1,2}(z)\in H^{1,1}(\mathbb{R})$, we have
\begin{align}
    &\sup_{x\in (-\infty,x_0)} \| \langle x \rangle \mathcal{P}^+(r_{2}e^{2izx})\|_{L^2_z} \leq c\| r_{2}\|_{H^1},\label{3.41}\\
    &\sup_{x\in (-\infty,x_0)} \| \langle x \rangle \mathcal{P}^-(r_{1}e^{-2izx})\|_{L^2_z} \leq c\| r_{1}\|_{H^1},\label{3.42}
\end{align}
\begin{align}
    &\sup_{x\in (-\infty,x_0)} \| \langle x \rangle \mathcal{P}^+(zr_{2}e^{2izx})\|_{L^2_z} \leq c\| r_{2}\|_{H^{1,1}},\nonumber\\
    &\sup_{x\in (-\infty,x_0)} \| \langle x \rangle \mathcal{P}^-(zr_{1}e^{-2izx})\|_{L^2_z} \leq c\| r_{1}\|_{H^{1,1}},\nonumber
\end{align}
where $\langle x \rangle=(1+x^2)^{1/2}$. Moreover, if $r_{1,2}\in L^{2,3}(\mathbb{R})$, then we have
\begin{align}
    &\sup_{x\in \mathbb{R}} \|z^i\mathcal{P}^+(z^kr_{2}e^{2izx}) \|_{L^2_z}\leq \| z^{i+k}r_{2}\|_{L^2_z}, \label{163} \\
    &\sup_{x\in \mathbb{R}} \|z^i\mathcal{P}^-(z^krr_{1}e^{-2izx}) \|_{L^2_z}\leq \| z^{i+k}r_{1}\|_{L^2_z}, \label{164}
\end{align}
where   $ i, k=0,1,2,3, i+k\leq 3 $ and  $c$ is a constant that depends on $\|r_{1,2}\|_{L^{\infty}}$ .
\end{proposition}
\begin{proof}
By an analogous analysis in \cite{Pelinovsky}, we can easy get this proposition.
\end{proof}

In order to obtain  estimates on the vector columns $\mu_-(x;z)-e_1$ and $\nu_+(x;z)-e_2$
 that will be needed in the subsequent section, we rewrite
 functions $\mu_-(x;z)-e_1$ and $\nu_+(x;z)-e_2$ by (\ref{pleme}) as
\begin{equation}
\mu_-(x;z)-e_1=\mathcal{P}^-(r_{1}e^{-2izx}\nu_+(x;z))(z), z\in \mathbb{R}
\label{89}\end{equation}
and
\begin{equation}
\nu_+(x;z)-e_2  =\mathcal{P}^+({\sigma r_{2}}e^{2izx}\mu_-(x;z))(z), z\in \mathbb{R},
\label{91}
\end{equation}
where we have used the fact
\begin{equation}
    \begin{split}
        M_-S=&[\mu_-,\nu_-]\left( \begin{array}{cc}
        \sigma r_{1}r_{2} &\sigma r_{2}e^{2izx}\\
        r_{1}e^{-2izx} & 0
        \end{array}
        \right)
        = [r_{1}e^{-2izx}\nu_+,\sigma r_{2}e^{2izx}\mu_-].
    \end{split}
\label{3.53}\end{equation}

Define  a function
\begin{equation}
   N(x;z)=[\mu_-(x;z)-e_1, \nu_+(x;z)-e_2],\nonumber
\end{equation}
which satisfies
\begin{equation}
    \begin{split}
        N-\mathcal{P}^+(NS_+)-\mathcal{P}^-(NS_-)=F,
    \end{split}
\label{1120}\end{equation}
where
\begin{equation}
    \begin{split}
        &F(x;z)=[\mathcal{P}^-(r_{1}e^{-2izx})e_2,\mathcal{P}^+(\sigma r_{2}e^{2izx})e_1],\\
        &S_+(x;z)=\left(\begin{array}{cc}
        0 & \sigma r_{2}e^{2izx}\\
        0 & 0
        \end{array}\right),\quad S_-(x;z)=\left(\begin{array}{cc}
        0 & 0\\
       r_{1}e^{-2izx} & 0
        \end{array}\right).
    \end{split}
\nonumber\end{equation}
\begin{lemma}
Let $r_{1,2}\in H^{1,1}(\mathbb{R})$, then for every $x_0\in \mathbb{R}^-$, we have
\begin{align}
    &\sup_{x\in (-\infty,x_0)} \| \langle x \rangle \mu_-^{(2)}(x;z)\|_{L^2_z}\leq c\| r_{1}\|_{H^{1}},\label{3.56}\\
    &\sup_{x\in (-\infty,x_0)} \| \langle x \rangle \nu_+^{(1)}(x;z)\|_{L^2_z}\leq c\| r_{2}\|_{H^{1}},\label{3.57}
\end{align}
\begin{align}
&\sup_{x\in (-\infty,x_0)}  \| \langle x \rangle \partial_x \mu_-^{(2)}(x;z)\|_{L^2_z}\leq c\|r_{1} \|_{ H^{1,1}}\label{172} \\
&\sup_{x\in (-\infty,x_0)}  \| \langle x \rangle \partial_x \nu_+^{(1)}(x;z)\|_{L^2_z}\leq c\|r_{2}\|_{ H^{1,1}}\label{173},
\end{align}
where $c$ is a constant that only depends on $\|r_{1,2}\|_{L^{\infty}}$.
In addition, if $r\in L^{2,3}$, then we have
\begin{align}
&\sup_{x\in \mathbb{R}} \|  z^{k}\partial^{i}_x \mu_-^{(2)}(x;z)\|_{L^2_z}\leq c\|r_{1} \|_{ L^{2,i+k}} \label{178}\\
&\sup_{x\in \mathbb{R}} \|  z^{k}\partial^{i}_x \nu_+^{(1)}(x;z)\|_{L^2_z}\leq c\| r_{2} \|_{ L^{2,i+k}}, \label{179}
\end{align}
where   $ i, k=0,1,2,3, i+k\leq 3 $ and  $c$ is a constant that depends on $\|r_{1,2}\|_{L^{\infty}}$ .
\end{lemma}
\begin{proof}
Note  $\mathcal{P}^+-\mathcal{P}^-=I$ and
$S_++S_-=(I-S_+)S$,
Eq.(\ref{1120}) can be rewritten as
\begin{equation}
    G-\mathcal{P}^-(GS)=F
\label{1125}\end{equation}
with $G=N(I-S_+)$. And the matrix $G(x;z)$ is written component-wise as
\begin{equation}
    G(x;z)=\left( \begin{array}{cc}
     \mu_-^{(1)}(x;z)-1    &  \nu_+^{(1)}-\sigma r_{2}e^{2izx}(\mu_-^{(1)}(x;z)-1)\\
    \mu_-^{(2)}(x;z)     &  \nu_+^{(2)}-1-\sigma r_{2}e^{2izx}\mu_-^{(2)}(x;z)
    \end{array}\right).\nonumber
\end{equation}
Comparing the second row of $F(x;z)$ with $G(x;z)$ and utilizing the bound (\ref{3.42}), we have
\begin{equation}
\begin{split}
    &\sup_{x\in (-\infty,x_0)}\|\langle x \rangle \mu_-^{(2)} \|_{L^2_z}\leq c\sup_{x\in (-\infty,x_0)}\|\langle x\rangle \mathcal{P}^{-}(r_{1}e^{-2izx}) \|_{L^2_z},\\
   &\|\nu_+^{(2)}-1- r_{1}e^{-2izx}\mu_-^{(2)}(x;z) \|_{L^2_z}\leq c\|\mathcal{P}^{-}(r_{1}e^{-2izx}) \|_{L^2_z},
\end{split}\label{3.64}
\end{equation}
where $c$ is a constant that depends on $\|r \|_{L^{\infty}}$. Substituting the bound (\ref{3.42}) into (\ref{3.64}), we obtain the estimate  (\ref{3.56}).

Similarly, comparing the first row of $F(x;z)$ and $G(x;z)$ yields
\begin{equation}
    \begin{split}
        &\|\mu_-^{(1)}(x;z)-1 \|_{L^2_z}\leq c\|\mathcal{P}^+(\sigma r_{2}e^{2izx}) \|_{L^2_z}, \\
        &\|\nu_+^{(1)}(x;z)-\sigma r_{2}e^{2izx}(\mu_-^{(1)}(x;z)-1) \|_{L^2_z}\leq c\|\mathcal{P}^+(\sigma r_{2}e^{2izx}) \|_{L^2_z}.
    \end{split}\label{1288}
\end{equation}

Taking derivative in x of (\ref{1120}), we obtain
\begin{equation}
    \partial_x N-\mathcal{P}^+(\partial_x N)S_+-\mathcal{P}^-(\partial_x N)S_-= F_{1}
\label{136}\end{equation}
with
\begin{equation}
    \begin{split}
        F_{1}=&\partial_x F+\mathcal{P}^+N\partial_xS_++\mathcal{P}^-N\partial_xS_- \\
        =&2i[e_2\mathcal{P}^-(-zr_{1}e^{-2izx}),e_1\mathcal{P}^+(z\sigma r_{2}e^{2izx})] \\
        &+2i\left(\begin{array}{cc}
       \mathcal{ P}^-(-zr_{1}(z)e^{-2izx}\nu_+^{(1)}(x;z)) &\mathcal{P}^+(z\sigma r_{2}e^{2izx}(\mu_-^{(1)}(x;z)-1)) \\
        \mathcal{P}^-(-zr_{1}e^{-2izx}(\nu_+^{(2)}(x;z)-1)) & \mathcal{P}^+(z\sigma r_{2}e^{2izx}\mu_-^{(2)}(x;z))
        \end{array}\right).
    \end{split}    \nonumber
\end{equation}
According to the estimates (\ref{163}) and (\ref{164}), together with the triangle equality, we obtain
\begin{align}
&z(\mu_-(x;z)-e_1)\in L_x^{\infty}((-\infty,x_0);L^2_z(\mathbb{R})),\nonumber\\
&z(\nu_+(x;z)-e_2)\in L_x^{\infty}((-\infty,x_0);L^2_z(\mathbb{R})).\nonumber
\end{align}
On account of the bound (\ref{3.41}), (\ref{3.42}) and
$r_{1,2}(z)\in L^{\infty}(\mathbb{R})$ , we conclude that $ F_{1}$ belongs to $L_x^{\infty}((-\infty,x_0);L^2_z(\mathbb{R}))$.
On account of the bound (\ref{3.56}), (\ref{3.57}) and
$r_{1,2}(z)\in L^{\infty}(\mathbb{R})$ , we conclude that $ \sup_{x\in (-\infty,x_0)}\langle x \rangle F_{1}$ belongs to $L_x^{\infty}((-\infty,x_0);L^2_z(\mathbb{R}))$. We get  (\ref{172}) and  (\ref{173}).

Taking  $j$-order  derivative     of (\ref{1120}), we obtain
\begin{equation}
    \partial^{j}_x N-\mathcal{P}^+(\partial^{j}_x N)S_+-\mathcal{P}^-(\partial^{j}_x N)S_-=F_{j},
\label{137}\end{equation}
with
$$
F_{j}=\partial_x F_{j-1}+\mathcal{P}^+\partial^{j-1}_xN_{x}\partial_xS_++\mathcal{P}^-\partial^{j-1}_xN_{x}\partial_xS_-.
$$
Where $j=2, 3$.
Repeating the analysis for (\ref{1125}),and using  (\ref{163}), (\ref{164})  we derive the estimates  (\ref{178}), (\ref{179}).
\end{proof}

\section{Reconstruction and estimates of the potential} \label{sec4}

We shall now recover the potential $u$  from the matrices $M_{\pm}$, which satisfy the integral equations (\ref{79}). This will gives us the map
$$
 H^{1,1}(\mathbb{R}) \cap L^{2,3}(\mathbb{R}) \ni r_{1,2} \mapsto u \in   H^{1,1}(\mathbb{R})\cap H^3(\mathbb{R}).
$$

\subsection{Estimates on the negative half-line}
Recalling (\ref{3.9})  which show that
\begin{equation}
u(x)=2iz\lim_{z\rightarrow \infty}(M_{\pm})_{12}.
\label{4.1}\end{equation}
We will use (\ref{4.1}) to get estimates of $u$ on the negative half-line. It follows from  (\ref{79}) and (\ref{3.9})  that
\begin{equation}
    u(x)=2i\lim_{|z|\rightarrow\infty}z\mathcal C((M_-S)_{12} .\label{136}
\end{equation}
 Since $r\in H^{1,1}(\mathbb R)\cap L^{2,3}(\mathbb R)$, we have $S(x;\cdot)\in L^1(\mathbb{R})\cap L^2(\mathbb{R})$. Besides,  the estimate (\ref{3.42}) implies that $M_-(x;\cdot)-I\in L^2(\mathbb{R})$. Therefore, we arrive at $L^1(\mathbb{R})$. Subsequently, Applying (\ref{3.14}) to (\ref{136}), we obtain
\begin{equation}
    \begin{split}
       u(x)&=\frac{1}{\pi}\int_{\mathbb{R}}\sigma r_{2}e^{2izx} \mu_-^{(1)}(x;z)dz\\
        &=\frac{1}{\pi}\int_{\mathbb{R}}\sigma r_{2}e^{2izx}P^{-}(re^{-2izx}\nu_+^{(1)})\sigma r_{2}e^{2izx}dz+\frac{1}{\pi}\int_{R}\sigma r_{2}e^{2izx}dz,
    \end{split}\label{4.4}
\end{equation}
where we have used the identity $\mu_-^{(1)}(x;z)-1=P^{-}(r_{1}e^{-2izx}\nu_+^{(1)})$ .
\begin{lemma} \label{3.6}
Let $r_{1,2}(z)\in H^{1,1}(\mathbb{R}) \cap L^{2,3}(\mathbb{R})$ , then
$u \in  H^{1,1}(\mathbb{R})\cap H^3(\mathbb{R})$. Moreover, we have
\begin{equation}
    \| u\|_{ H^{1,1}(\mathbb{R}^{-})\cap H^{3}(\mathbb{R}^{-})}\leq c (\| r_{1}\|_{ H^{1,1}(\mathbb{R})\cap L^{2,3}(\mathbb{R})}+\| r_{2}\|_{ H^{1,1}(\mathbb{R})\cap L^{2,3}(\mathbb{R})}),
\label{104}\end{equation}
where $c$ is a constant that  depends on $\| r_{1,2}\|_{L^{\infty}}$ and $\|zr_{1,2}\|_{L^{\infty}}$.
\end{lemma}
\begin{proof}
Recalling (\ref{4.4}) and the results from the Fourier theory. For a function $r_{1,2}(z)\in L^2(\mathbb{R})$, by Parseval's equation, we have
\begin{equation}
    \| r_{1,2}\|_{L^2}=\|\hat{r}_{1,2} \|_{L^2},\nonumber
\end{equation}
where the function $\hat{r}_{1,2}$ denotes the Fourier transform.
Since $L^{2,3}(\mathbb{R})\cap H^{1,1}(\mathbb{R})$, the second term of (\ref{4.4}) belongs to $ H^{1,1}(\mathbb{R})\cap H^{3}(\mathbb{R})$  due to the property $\widehat{\partial_z r_{1,2}(z)}=x\hat{r}_{1,2}(x)$.

Let
\begin{equation}
    \uppercase\expandafter{\romannumeral1}(x)=\int_{\mathbb{R}}\sigma r_{2}e^{2izx} ( \mu_-^{(1)}(x;z)-1 ) dz.\label{4.99}
\end{equation}
Substituting (\ref{91}) into (\ref{4.99}) and applying the Fubini's theorem yields
\begin{equation}
    \begin{split}
       \uppercase\expandafter{\romannumeral1}(x)=&\int_{\mathbb{R}} \sigma r_{2}e^{2izx} \lim_{\epsilon\rightarrow 0} \frac{1}{2\pi i}\int_{\mathbb{R}}\frac{r_{1}(s)e^{-2isx}\nu_+^{(1)}(s)}{s-(z-i\epsilon)}dsdz\\
       =&-\int_{\mathbb{R}}r_{1}(s) e^{-2isx} \nu_+^{(1)}(s)\lim_{\epsilon\rightarrow 0}\frac{1}{2\pi i}\int_{\mathbb{R}}\frac{\sigma r_{2}e^{2izx}}{z-(s+i\epsilon)} dzds\\
       =& -\int_{\mathbb{ R}} r_{1}(z)e^{-2izx}\nu_+^{(1)}(z) P^+(\sigma r_{2}e^{2isx})(z)dz.
    \end{split}\nonumber
\end{equation}
 Therefore, for every $x_0\in \mathbb{R}^-$, utilizing the H\"{o}lder's inequality and  the estimates (\ref{3.41}) and (\ref{3.57}), we find
\begin{equation}
    \begin{split}
      &  \sup_{x\in (-\infty,x_0)}\left| \langle x\rangle^2\uppercase\expandafter{\romannumeral1}(x) \right|\leq  \| r_{1}\|_{L^{\infty}} \sup_{x\in (-\infty,x_0)} \|\langle x\rangle\nu_+^{(1)} \|_{L^2}\\
        &\qquad\times\sup_{x\in (-\infty,x_0)} \|\langle x\rangle P^+(\sigma r_{2}e^{2isx}) \|_{L^2}
        \leq   c\|r_{1} \|_{H^1}\|r_{2} \|_{H^1},
    \end{split}
\label{4.9}\end{equation}
where $c$ is a constant only depends on $\|r \|_{L^{\infty}}$,
further, we obtain
\begin{equation}
    \begin{split}
        \|\langle x\rangle I(x) \|_{L^2(\mathbb{R}^-)}\leq   c\|r_{1} \|_{H^1}\|r_{2} \|_{H^1},
    \end{split}\nonumber
\end{equation}
where $c$ is another constant only depends on $\|r \|_{L^{\infty}}$.
 Combining the results of the two terms of (\ref{4.9}) leads to
 \begin{equation}
     \|u(x)\|_{L^{2,1}( \mathbb{R}^-)}\leq c(1+\| r_{1}\|_{H^1}+\| r_{2}\|_{H^1})(\| r_{1}\|_{H^1}+\| r_{2}\|_{H^1}).
\label{111} \end{equation}
 This completes  the proof of $u\in L^{2,1}(\mathbb{R}^-)$.
 By the Fourier theory, the derivative of the second term of (\ref{4.4}) belongs to $L^2(\mathbb{R})$. For the second term $I(x)$, we differentiate $I(x)$ in $x$ and obtain
 \begin{equation}
     \begin{split}
           I'(x)=& \partial\int_{\mathbb{R}} \sigma r_{2}e^{2izx}(\mu_-^{(1)}(x;z)-1) dz\\
         =&-2i\int_{\mathbb{R}} r_{1}(z) e^{-2izx}\nu_+^{(1)}(x;z)P^+(s\sigma r_{2}e^{2isx})(z)dz\\
         &-2i\int_{\mathbb{R}} zr_{1}(z)e^{-2izx}\nu_+^{(1)}(x;z)P^+(\sigma r_{2}e^{2isx})(z)dz\\
         &-\int_{\mathbb{R}} r_{1}(z) e^{-2izx} \partial_x \nu_+^{(1)}(x;z)P^+(\sigma r_{2}e^{2isx})(z)dz,
     \end{split}\nonumber
 \end{equation}
 where we have used Eq.(\ref{89}) and  the Fubini's theorem.

 Utilizing the estimates (\ref{3.41}),  (\ref{3.57})  we find that for every $x_0\in \mathbb{R}^- $,
 \begin{equation}
     \begin{split}
       &  \sup_{x\in (-\infty,x_0)}|\langle x\rangle^{2}   I'(x)  | \leq
 \| r\|_{L^{\infty}} \sup_{x\in (x_0,+\infty)}  \left ( 2 \|\langle x\rangle\nu_+^{(1)} \|_{L^2}  \| \langle x\rangle  P^+(z\overline{r}e^{2izx})\|_{L^2}\right.\\
         &\left.+ \|\langle x\rangle\nu_+^{(1)} \|_{L^2}  \| \langle x\rangle P^+(z\sigma r_{2}e^{2izx})\|_{L^{2}} +  \|\langle x\rangle \partial_x\nu_+^{(1)} \|_{L^2} \|\langle x\rangle P^+(\sigma r_{2}e^{2izx})\|_{L^2}\right) \\
         & \leq   c\|r_{1}\|_{H^{1,1}\cap L^{2,3}}\|r_{2}\|_{H^{1,1}\cap L^{2,3}},
         \end{split}\nonumber
\end{equation}
which implies that
\begin{equation}
\begin{split}
\| \langle x\rangle  I'(x) \|_{L^2(\mathbb{R}^-)}
\leq & c\|r_{1}\|_{H^{1,1}\cap L^{2,3}}\|r_{2}\|_{H^{1,1}\cap L^{2,3}},
\label{1113}\end{split}
\end{equation}
with $c$ is another constant that depends on $\| r\|_{L^{\infty}}$ .
Subsequently, we obtain  $  I'(x)\in L^{2,1}(\mathbb{R}^-)$. We conclude that $u\in H^{1,1}(\mathbb{R}^-)$. The estimate (\ref{104}) can be obtained from
(\ref{111}) and (\ref{1113}) with another constant $c$ depending on $\| r\|_{L^{\infty}}$.

For $  I''(x)$,  we get
\begin{equation}
     \begin{split}
      &   \sup_{x\in (-\infty,x_0)}|\langle x\rangle   I''(x)  | \leq
 \| r_{1}\|_{L^{\infty}} \sup_{x\in (x_0,+\infty)}  \left (  4 \|\langle x\rangle\nu_+^{(1)} \|_{L^2}  \| z^{2}P^+(\sigma r_{2}e^{2izx})\|_{L^2}\right.\\
         &+ \|\langle x\rangle\nu_+^{(1)} \|_{L^2}  \|\langle x\rangle P^+(z^{2}\sigma r_{2}e^{2izx})\|_{L^2}
         + \|\partial^{2}_x\nu_+^{(1)} \|_{L^2}  \|\langle x\rangle P^+(\sigma r_{2}e^{2izx})\|_{L^2}\\
         &+2 \|z\nu_+^{(1)} \|_{L^2}  \|\langle x\rangle P^+(z\sigma r_{2}e^{2izx})\|_{L^{2}} + \|\langle x \rangle \partial_x\nu_+^{(1)} \|_{L^2} \|z P^+(\sigma r_{2}e^{2izx})\|_{L^2}\\
         &\left.+2 \|\partial_x\nu_+^{(1)} \|_{L^2} \|\langle x \rangle P^+(s\sigma r_{2}e^{2izx})\|_{L^{2}} \right)
         \leq   c\|r_{1}\|_{H^{1,1}\cap L^{2,3}}\|r_{2}\|_{H^{1,1}\cap L^{2,3}}.
         \end{split}\nonumber
\end{equation}
For $  I'''(x)$ similar to above analysis, we get
\begin{equation}
     \begin{split}
         &\sup_{x\in (-\infty,x_0)}|\langle x\rangle  I'''(x) | \leq \| r_{1}\|_{L^{\infty}}   \sup_{x\in (x_0,+\infty)}  \left(   8 \|\langle x\rangle\nu_+^{(1)}  \|_{L^2} \| P^+(z^{3}\sigma r_{2}e^{2izx})\|_{L^2}\right.\\
         &+ \|\langle x\rangle\nu_+^{(1)}  \|_{L^2}  \|  z^{3}P^+(\sigma r_{2}e^{2izx})\|_{L^2}+ \|\partial^{3}_x\nu_+^{(1)}  \|_{L^2} \|\langle x\rangle P^+(\sigma r_{2}e^{2izx})\|_{L^2}\\
         &+4 \|\langle x\rangle\partial_x\nu_+^{(1)} \|_{L^2} \| P^+(s^{2}\sigma r_{2}e^{2izx})\|_{L^2}+2 \|\langle x\rangle\partial^{2}_x\nu_+^{(1)} \|_{L^2}  \| P^+(z^{2}\sigma r_{2}e^{2izx})\|_{L^2}\\
         &+2 \|\langle x\rangle\nu_+^{(1)}  \|_{L^2}  \| z^{2}P^+(s\sigma r_{2}e^{2izx})\|_{L^2}+4 \|\langle x\rangle\nu_+^{(1)}  \|_{L^2}  \|z P^+(z^{2}\sigma r_{2}e^{2izx})\|_{L^2}\\
         &+ \|\partial_x\nu_+^{(1)}\|_{L^2}  \|\langle x\rangle s^{2}P^+(\sigma r_{2}e^{2izx})\|_{L^2}+ \| z\partial^{2}_x\nu_+^{(1)}\|_{L^2} \|\langle x\rangle P^+(\sigma r_{2}e^{2izx})\|_{L^2}\\
         &\left.+4 \|\langle x\rangle\nu_+^{(1)}\|_{L^2_z} \| P^+(s^{2}\sigma r_{2}e^{2izx})\|_{L^2}+2 \|z\partial_x\nu_+^{(1)}\|_{L^2}  \|\langle x\rangle P^+(z\overline{r}e^{2izx})\|_{L^2} \right)\\
          &\leq  c\|r_{1}\|_{H^{1,1}\cap L^{2,3}}\|r_{2}\|_{H^{1,1}\cap L^{2,3}}.
         \end{split}\nonumber
\end{equation}
Then we conclude that $u\in H^{3}(\mathbb{R}^-)$. Finally we  prove the conclusion.
\end{proof}

We actually get the following map through above analysis.
$$
 H^{1,1}(\mathbb{R}) \cap L^{2,3}(\mathbb{R}) \ni r_{1,2} \mapsto u \in H^{3}(\mathbb{R}^{-}) \cap H^{1,1}(\mathbb{R}^{-}) .
$$
We will prove the map is Lipschitz continuous.
\begin{lemma} Let $r\in H^{1,1}(\mathbb{R}) \cap L^{2,3}(\mathbb{R})$ , then the mapping
\begin{equation}
 H^{1,1}(\mathbb{R}) \cap L^{2,3}(\mathbb{R}) \ni r_{1,2} \mapsto u \in H^{3}(\mathbb{R}^{-}) \cap H^{1,1}(\mathbb{R}^{-})\nonumber
\end{equation}
is Lipschitz continuous.
\label{lemma15}\end{lemma}
\begin{proof}
Let $r_{1,2},\tilde{r}_{1,2} \in  H^{1,1}(\mathbb{R}) \cap L^{2,3}(\mathbb{R})$. Let the functions $u$ and $\tilde u$ are the corresponding potentials respectively.  We will show that there exists a constant $c$ that  depends on $\| r_{1,2}\|_{L^{\infty}}$ such that
\begin{equation}
    \|u-\tilde u \|_{H^{3}(\mathbb{R}^{-}) \cap H^{1,1}(\mathbb{R}^{-})} \leq c(\|r_{1}-\tilde r_{1} \|_{ H^{1,1}(\mathbb{R}) \cap L^{2,3}(\mathbb{R})}+\|r_{2}-\tilde r_{2} \|_{ H^{1,1}(\mathbb{R}) \cap L^{2,3}(\mathbb{R})}). \label{wiw}
\end{equation}
From (\ref{4.9}), we have
\begin{equation}
\begin{split}
    u-\tilde u=&\frac{1}{\pi}\int_{\mathbb{R}}(\sigma r_{2}-\sigma\tilde r_{2} )e^{2izx}d z+ \frac{1}{\pi}\int_{\mathbb{R}}(\sigma r_{2}-\sigma\tilde  r_{2})e^{2izx}  (\mu_-^{(1)}(x;z)-1) d z\\
    &+\frac{1}{\pi}\int_{R} \sigma r_{2}e^{2izx}(\mu_-^{(1)}(x;z)-\tilde\mu_-^{(1)}(x;z)) d z.
\end{split}\nonumber
\end{equation}
Repeating the analysis in the proof of Lemma \ref{3.6}, we obtain the Lipschitz continuity of $u$.
\end{proof}
\subsection{Estimates on the positive half-line}
Recalling (\ref{3.9})  again, we can get
\begin{equation}
u(-x)=2iz\sigma\lim_{z\rightarrow \infty}(M_{\pm})_{21}.
\end{equation}

Performing the same manipulation
for (\ref{4.1})
yields
\begin{equation}
    \begin{split}
        u(- x)=&-\frac{\sigma}{\pi}\int_{\mathbb{R}} r_1(z)e^{-2izx}( \nu_-^{(2)}(x;z)+\sigma  r_2(z) e^{2ikx}\mu_-^{(2)}(x;z) )dz\\
        =&-\frac{\sigma}{\pi}\int_{\mathbb{R}}r_1(z)e^{-2izx}\nu_+^{(2)}(x;z)dz,
    \end{split}\label{134}
\end{equation}
where we have used the identity $\nu_+^{(2)}=\sigma r_2(z)e^{2izx}\mu_-^{(2)}+\nu_-^{(2)}$.

Similar to condition $x_{0}\in \mathbb{R}^{-}$, we summarize the above analysis as the following lemma:
\begin{lemma}\label{4.8}
Let $r_{1,2}(z)\in H^{3}(\mathbb{R})\cap H^{1,1}(\mathbb{R})$ satisfy $|r_{1,2}(z)|<1$, then
$u\in H^{1,1}(\mathbb{R}^{+})\cap L^{2,3}(\mathbb{R}^{+})$. Moreover, we have
\begin{equation}
    \| u\|_{ H^{1,1}(\mathbb{R}^{+})\cap L^{2,3}(\mathbb{R}^{+})}\leq c(  \| r_1\|_{H^{3}(\mathbb{R})\cap H^{1,1}(\mathbb{R})}+\|r_2\|_{H^{3}(\mathbb{R})\cap H^{1,1}(\mathbb{R})}    ),
\end{equation}
where c is a constant that depends on $\| r_{1,2}\|_{L^{\infty}}$ and $\|zr_{1,2}\|_{L^{\infty}}$.
\end{lemma}
\begin{proof}
We rewrite (\ref{134})  as
\begin{equation}
    u(-x)=-\frac{\sigma}{\pi}\int_{\mathbb{R}} r_1(z)e^{-2izx}  dz-\frac{\sigma}{\pi}\int_{\mathbb{R}} r_1(z)e^{-2izx} (\nu_+^{(2)}(x;z)-1) dz.
\end{equation}
Let
\begin{equation}
    \hat r_1(-x)=\int_{\mathbb{R}} r_1(z)e^{-2iz(-x)}dz.
\end{equation}
According to the Fourier theory, we have
$-x\hat r_1(-x)=\widehat{\partial_zr_1(z)}(x)$
and
$\|x\hat r_1(-x)\|_{L^2(\mathbb{R})}=\| \partial_z r_1(z)\|_{L^2(\mathbb{R})}$.
Let
$$\uppercase\expandafter{\romannumeral1}_2(x)=-\frac{\sigma}{\pi}\int_{\mathbb{R}}r_1(z)e^{-2izx}(\nu_+^{(2)}(x;z)-1)dz.$$
Repeating the analysis in the proof of Lemma \ref{3.6}, we obtain $u\in H^{1,1}(\mathbb{R}^+)\cap L^{2,3}(\mathbb{R}^+)$ and
\begin{equation}
    \| u\|_{H^{1,1}(\mathbb{R}^+)\cap L^{2,3}(\mathbb{R}^+}\leq c(  \| r_1\|_{H^{3}(\mathbb{R})\cap H^{1,1}(\mathbb{R})}+\|r_2\|_{H^{3}(\mathbb{R})\cap H^{1,1}(\mathbb{R})}   ),
\end{equation}
where c is a constant that depends on $\| r_{1,2}\|_{L^{\infty}}$ and $\|zr_{1,2}\|_{L^{\infty}}$.
\end{proof}
\begin{lemma} Let $r_{1,2}(z)\in H^{1,1}(\mathbb{R}) \cap L^{2,3}(\mathbb{R})  $ , then the following map
\begin{equation}
H^{1,1}(\mathbb{R}) \cap L^{2,3}(\mathbb{R}) \ni r_{1,2} \mapsto u \in H^{3}(\mathbb{R}^{+}) \cap H^{1,1}(\mathbb{R}^{+})\nonumber
\end{equation}
is Lipschitz continuous.
\label{lemma16}\end{lemma}
 Summarize the results from Lemma \ref{lemma15} to Lemma \ref{lemma16}, we have the following proposition:

\begin{proposition} \label{prop4.1}
Let $r_{1,2}(z)\in H^{1,1}(\mathbb{R}) \cap L^{2,3}(\mathbb{R})$, then we have
$u \in H^{3}(\mathbb{R}) \cap H^{1,1}(\mathbb{R}) $ and
\begin{equation}
    \| u\|_{H^{3}(\mathbb{R}) \cap H^{1,1}(\mathbb{R}) }\leq c  (\| r_{1}\|_{H^{1,1}(\mathbb{R}) \cap L^{2,3}(\mathbb{R})}+\| r_{2}\|_{H^{1,1}(\mathbb{R}) \cap L^{2,3}(\mathbb{R})}).\nonumber
\end{equation}
Moreover, the mapping
\begin{equation}
H^{1,1}(\mathbb{R}) \cap L^{2,3}(\mathbb{R}) \ni r_{1,2} \mapsto u \in H^{3}(\mathbb{R}) \cap H^{1,1}(\mathbb{R})\nonumber
\end{equation}
is Lipschitz continuous.
\end{proposition}

\section{ Global existence and Lipschitz continuity }  \label{sec5}

\subsection{Time evolution of scattering data }
 From Section 2  to  Section 4,   for the  initial data   $u(x,0)\in H^3(\mathbb{R}) \cap H^{1,1}(\mathbb{R})$,  we only  consider the spatial spectral problem
 (\ref{lax1}) and
 obtain its unique normalized  solution
\begin{align}
  & m_{1}^{\pm}(x,0; z)\rightarrow e_1,  \  \   m_{1}^{\pm}(x,0; z)\rightarrow e_2,   \ \ x \rightarrow  \pm\infty,  \label{pw1}
\end{align}
which cannot satisfy the time spectral problem  (\ref{lax2})  since they  are short of a function about the  time $t$.
For  every $t\in [0,T]$, we define  the normalized Jost functions  of the Lax pair (\ref{lax1}) and (\ref{lax2}):
\begin{align}
&m_{1}^{\pm}(x,t; z)=m_{1}^{\pm}(x,0; z)e^{4iz^3t}, \label{time1}\\
 &m_{2}^{\pm}(x,t; z)=m_{2}^{\pm}(x,0; z)e^{-4iz^3t},\label{time2}
\end{align}
with the potential $u(\cdot,0)\in H^3(\mathbb{R}) \cap H^{1,1}(\mathbb{R}) $.
It follows that for every $t\in[0,T]$, we have
$$m_{1}^{\pm}(x,t; z)\rightarrow e^{4iz^3t}e_1 \quad x\rightarrow\pm\infty,$$
$$m_{2}^{\pm}(x,t; z)\rightarrow e^{-4iz^3t}e_2 \quad x\rightarrow\pm\infty.$$
Repeating the analysis as  the proof of Lemma\ref{lemma1}, we prove that there exist unique solutions of the Volttera's integral equations for  Jost functions $m_{1}^{\pm}(x,t; z)$ and $m_{2}^{\pm}(x,t; z)$, and the  Jost functions $m_{1}^{\pm}(x,t; z)$ and $m_{2}^{\pm}(x,t; z)$ admit the same
 analytic property as $m_{1}^{\pm}(x,0; z)$ and $m_{2}^{\pm}(x,0; z)$.
 As well,   for every $(x,t)\in \mathbb{R}\times \mathbb{R}^+$ and every $z\in \mathbb{R}$,  the Jost functions $m_{1}^{\pm}(x,t; z)$ and $ m_{2}^{\pm}(x,t; z)$  are supposed to satisfy  the  scattering  relation
\begin{equation}
\begin{split}
    &m_{1}^{+}(x,t; z)=a(t;z)m_{1}^{-}(t,x; z)+b(t;z)e^{-2izx }m_{2}^{-}(t,x; z),\\
    &m_{2}^{+}(x,t; z)=c(t;z)e^{2izx }m_{1}^{-}(t,x; z+d(t;z)m_{2}^{-}(t,x; z).
\end{split}
\nonumber\end{equation}
By the Crammer's law and  the evolution relation  (\ref{time1})-(\ref{time2}),  we obtain the evolution   of   the scattering coefficients
\begin{equation}
\begin{split}
&a(t;z)=W(m_{1}^{+}(0,0; z)e^{4iz^3t},m_{2}^{-}(0,0; z)e^{-4iz^3t})=a(0;z),\\
&d(t;z)=W(m_{1}^{-}(0,0; z)e^{4iz^3t},m_{2}^{+}(0,0; z)e^{-4iz^3t})=d(0;z),\\
&b(t;z)=W(m_{1}^{-}(0,0; z)e^{-4iz^3t},m_{1}^{+}(0,0; z)e^{-4iz^3t}) = b(0;z) e^{8iz^3t}.
\end{split}\nonumber
\end{equation}

Direct calculation shows that  the reflection coefficients are given by
\begin{align}
&r_{1}(t;z)=\frac{b(t;z) }{a(t;z)}=\frac{b(0;z)}{a(0;z)}e^{8iz^3t}=r_{1}(0;z)e^{8iz^3t},\\
&r_{2}(t;z)=\frac{\bar b(t;-z) }{d(t;z)}=\frac{\bar b(0;-z)}{d(0;z)}e^{8iz^3t}=r_{2}(0;z)e^{-8iz^3t},
\label{eeew}
\end{align}
where $r_{1,2}(0;z)$ are initial reflection data founded from the initial data $u(x,0)$.
\begin{proposition}
If $r_{1,2}(0;z)\in  H^{1,1}(\mathbb{R}) \cap L^{2,3}(\mathbb{R})$, then for any  fixed $T>0$ and   $t\in[0,T]$ or $t\in[-T,0]$, we have $r(t;z)\in H^{1,1}(\mathbb{R}) \cap L^{2,3}(\mathbb{R}).$
\label{prop5.1}\end{proposition}

\begin{proof}
By (\ref{eeew}), we obtain
\begin{equation}
\| r_{1}(t;\cdot) \|_{L^{2,3}(\mathbb{R})}=\| r_{1}(0;\cdot) \|_{L^{2,3}(\mathbb{R})}.\nonumber
\end{equation}

For every $t\in[0,T]$, we have
\begin{equation}
\begin{split}
 \| z\partial_zr_{1}(t;\cdot) \|_{L^{2}(\mathbb{R})}=&\|z\partial_z r_{1}(0;\cdot)+24iz^{3}tr_1(t;z) \|_{L^2(\mathbb{R})}\\
\leq & \|z\partial_z r_{1}(0;\cdot) \|_{L^2(\mathbb{R})}+24T\|r_{1}(0;\cdot) \|_{L^{2,3}(\mathbb{R})}.
\end{split}\label{6.3}
\end{equation}
Therefore, we infer  that $r_{1}(t;\cdot)\in H^{1,1}(\mathbb{R}) \cap L^{2,3}(\mathbb{R})$ for every $t\in[0,T]$ as $r_{1}(0;\cdot)\in H^{1,1}(\mathbb{R}) \cap L^{2,3}(\mathbb{R}).$ We can get similar conclusion for $t\in[-T,0]$ and $r_{2}$.
\end{proof}
Using the time-dependent data $r_{1,2}(t;z)$ we can construct a time-dependent RH-problem
\begin{problem}\label{problem2}
 Find a   matrix function  $M(x,t; z)$ satisfying

(i)  $M(x ,t; z) \rightarrow I+\mathcal{O}\left(z^{-1}\right)$ as $z \rightarrow \infty$.

(ii)   For   $  M(x ,t; z)$ admits the following jump condition
\begin{equation}
\begin{aligned}
M_{+}(x,t; z) &   =M_{-}(x,t; z) V_{x,t}(z), \label{prrwe}
\end{aligned}
\end{equation}
where
\begin{equation}
V_{x,t}(z):=\left(\begin{array}{cc}
1+\sigma r_{1}r_{2} & \sigma r_{2} e^{2 \mathrm{i} \theta(x,t;z)}\\
r_{1}e^{-2 \mathrm{i} \theta(x,t;z)} & 1
\end{array}\right),z\in \mathbb{R}.
\end{equation}
Where $\theta(x,t;z)=zx+4z^{3}t$.
\end{problem}

\begin{theorem}
Assume that
$$
M(x,t; z)=\begin{cases}
m_{+}(x;z,t)& \operatorname{Im} z>0\\
m_{-}(x;z,t)&\operatorname{Im} z<0.
\end{cases}
$$
is the solution of Problem\ref{problem2}.
 Then $M(x,t;z)$ satisfies the following system od linear differential equations:
\begin{equation}
\begin{aligned}
&M_{x}(x,t;z)=iz[\sigma_{3},M]+QM,\\
&M_{t}(x,t;z)=4iz^{3}[\sigma_{3},M]+(4z^{2}Q-2iz(Q_{x}-Q^{2})\sigma_{3}+2Q^3-Q_{xx} )M,
\end{aligned}
\label{107}\end{equation}
where
\begin{equation}
Q(x,t)=\begin{bmatrix}
0&u(x,t)\\
-u(-x,-t)&0
\end{bmatrix},
\end{equation}
and
\begin{equation}
u(x,t)=2iz\lim_{z\rightarrow \infty}M_{12}(x,t).
\label{113}\end{equation}
\end{theorem}
\begin{proof}
Define
\begin{equation}
\begin{aligned}
&LM=M_{x}(x,t;z)-iz[\sigma_{3},M]-QM,\\
&NM=M_{t}(x,t;z)-4iz^{3}[\sigma_{3},M]-(4z^{2}Q-2iz(Q_{x}-Q^{2})\sigma_{3}+2Q^3-Q_{xx} )M.
\end{aligned}
\label{114}\end{equation}
Where $M(x,t;z)$ satisfies Problem\ref{problem2} and $u(x,t)=2iz\lim_{z\rightarrow \infty}M_{12}(x,t).$

Then, by direct calculation we know $LM$ and $NM$ satisfy
$$
(LM)_{+}=(LM)_{-}V_{x,t},\ (NM)_{+}=(NM)_{-}V_{x,t}.
$$
Using definition of Problem\ref{problem2} we can rewrite $M(x,t; z)$ in the following form
\begin{equation}
M(x,t; z)=I+\frac{M_1}{z}+\frac{M_2}{z^{2}}+\cdot\cdot\cdot+\frac{M_n}{z^n}+\cdot\cdot\cdot
\label{112}\end{equation}
Substitute (\ref{112}) into (\ref{114}) and compare the  coefficients of  $z$ on both side of equation and noticing (\ref{113}). We can get
$$
LM \sim \mathcal{O}(z^{-1}),\ NM \sim \mathcal{O}(z^{-1}).
$$
In other words, if $M(x,t;z)$ satisfies Problem\ref{problem2} and $u(x,t)=2iz\lim_{z\rightarrow \infty}M_{12}(x,t)$ then $LM$ and $NM$ satisfy problem\ref{problem2} and $LM \rightarrow 0,\ NM \rightarrow 0$ as $z\rightarrow 0$. Noticing Lemma\ref{lemma3.4} and proposition\ref{prop5.1} we know Problem\ref{problem2} has unique solution which yields $LM=NM=0$. Then we get (\ref{107}).
\end{proof}

So under the time evolution of the scattering data $r(x,t;z)$ in (\ref{eeew}), the function reconstructed from the Problem \ref{ope} through the reconstruction formula (\ref{4.4}) under time-dependent scattering data $r(t;z)$ is still a solution of the Mkdv equation (\ref{1}).

\subsection{The proof of main results }
In this section, we will prove the existence of the local  and  global solutions to the Cauchy problem .
The scheme behind the proof  can be described as below.
\begin{lemma} \label{leea}
   Let   the initial data  $u_0(x) \in H^{3}(\mathbb{R}) \cap H^{1,1}(\mathbb{R}) $ ,
then there exists a unique local solution  to the Cauchy problem (\ref{mkdv1})-(\ref{mkdv2}).
\begin{align}
u \in C([0,T],   H^{3}(\mathbb{R}) \cap H^{1,1}(\mathbb{R})).\nonumber
\end{align}
Furthermore, the map
\begin{equation}
  H^{3}(\mathbb{R}) \cap H^{1,1}(\mathbb{R}) \ni u_{0}  \mapsto u  \in C ([0, T],\in   H^{3}(\mathbb{R}) \cap H^{1,1}(\mathbb{R}))\nonumber
\end{equation}
is Lipschitz continuous.
\end{lemma}
\begin{proof}
Performing a similar analysis as Lemma \ref{lemma15} and lemma \ref{lemma16},  we can establish a RH problem for $r(t;z)$ for every $t\in[0,T]$ and address the existence and uniqueness of a solution to the RH problem. Further, the potential $u(t,x)$ can be recovered from the reflection coefficients $r(t;z)$. Moreover, the potential $u(t, \cdot)    \in H^{1,1}\cap H^{3} $ for every $t\in[0,T]$ and is Lipschitz continuous of $r(t;z)$. Thus we have
\begin{equation}
\begin{split}
    \| u(t;\cdot) \|_{H^{1,1}\cap H^{3} } &\leq c_1\|r(t;\cdot)\|_{\in  H^{1,1} \cap L^{2,3}}\\
    &\leq c_2r(0;\cdot)\|_{H^{1,1} \cap L^{2,3}}   \leq c_3 \|u_0 \|_{H^{1,1}\cap H^{3}}
\end{split}\label{158}
\end{equation}
where the positive constants $c_1$,$c_2$ and $c_3$ depends on $\|r\|_{L^{\infty}},\ \|zr\|_{L^{\infty}}$ and $(T,\ \|u_0 \|_{H^{1,1}(\mathbb{R})\cap H^{3}(\mathbb{R})})$, respectively.

Next we show $u(x ,t)$ is continuous  with respect to  $t\in[0,T]$ under the $H^{1,1}(\mathbb{R})\cap H^{3}(\mathbb{R})$ norm.
Let $t \in[0,T]$ and $|\Delta t|<1$ such that $t+\Delta t \in [0,T]$,  then with the Lipschitz continuity from $u(t,x)$ to $r(t;z)$ in Proposition \ref{prop4.1}, we
have
\begin{equation}
\begin{aligned}
&\|u(t+\Delta t, x  )-u( t, x  )\|_{H^{1,1}(\mathbb{R})\cap H^{3}(\mathbb{R})} \\
&\leq    c  (\| r_{1}(t+\Delta t; z )-r_{1}( t;z  ) \|_{H^{1,1}(\mathbb{R})\cap L^{2,3}(\mathbb{R})}+\| r_{2}(t+\Delta t; z )-r_{2}( t;z  ) \|_{H^{1,1}(\mathbb{R})\cap L^{2,3}(\mathbb{R})})\\
&\leq c|\Delta t|  (\|r_{1}(0;z)\|_{H^{1,1}(\mathbb{R})\cap L^{2,3}(\mathbb{R})}+\|r_{2}(0;z)\|_{H^{1,1}(\mathbb{R})\cap L^{2,3}(\mathbb{R})}) \leq c |\Delta t| \rightarrow 0, \ \Delta t \to 0,
\end{aligned}\nonumber
\end{equation}
which together with the estimate (\ref{158}) implies that there exists a unique local solution $u(x,t)\in C([0,T],H^{3}(\mathbb{R})\cap H^{1,1}(\mathbb{R}))$ to the Cauchy problem (\ref{mkdv1})-(\ref{mkdv2}) and the map
$$
H^{1,1}(\mathbb{R})\cap H^{3}(\mathbb{R}) \ni u_{0}(x) \mapsto u(t,x) \in C ([0, T], H^{1,1}(\mathbb{R})\cap H^{3}(\mathbb{R}))
$$
 is Lipschitz continuous.
 \end{proof}

Finally we  give the proof of Theorem \ref{th1}.

\begin{proof} [Proof of  Theorem \ref{th1}]

Suppose the maximal time in which the local solution in Lemma \ref{leea}  exists is $T_{max}$.

If $T_{max}=\infty$, then the local solution is global.

If the local solution exists in the closed interval $[0,T_{max}]$, we can use $u(T_{max},\cdot)\in H^{3}(\mathbb{R})\cap H^{1,1}(\mathbb{R})$ as a new initial data. By a similar analysis as the previous sections, there exists a positive constant $T_1$ such that the solution $u\in C([T_{max},T_{max}+T_1], H^{1,1}(\mathbb{R})\cap H^{3}(\mathbb{R}))$ exists. This contradicts with the maximal time assumption.

If the local solution exists in the open interval $[0,T_{max})$. According to (\ref{158}), we have
\begin{equation}
\|u(t,x) \|_{H^{1,1}(\mathbb{R})\cap H^{3}(\mathbb{R})}\leq c_3(T_{max})\|u_0 \|_{H^{1,1}(\mathbb{R})\cap H^{3}(\mathbb{R})},\quad t\in[0,T_{max}).\nonumber
\end{equation}
 Due to the continuity of $u(t,x)$ to the time $t$, the limit of $u(t,x)$ as $t$ approaches to $T_{max} $ exists. Let $ u_{max}(x):=\lim_{t\rightarrow T_{max}} u(t,x)$. Taking the limit by $t\rightarrow T_{max}$ in  (\ref{158}), we have
 \begin{equation}
 \|u_{max} \|_{H^{1,1}(\mathbb{R})\cap H^{3}(\mathbb{R})}\leq  c_3(T_{max})\|u_0 \|_{H^{1,1}(\mathbb{R})\cap H^{3}(\mathbb{R})},\nonumber
 \end{equation}
 which implies that we can extend the local solution $u\in C([0,T_{max}),H^{1,1}(\mathbb{R})\cap H^{3}(\mathbb{R}))$ to $u\in C([0,T_{max}],H^{1,1}(\mathbb{R})\cap H^{3}(\mathbb{R}))$, this contradicts with the premise that $[0,T_{max})$ is the maximal open interval.
\end{proof}

\noindent\textbf{Acknowledgements}

This work is supported by  the National Natural Science
Foundation of China (Grant No. 12271104,  51879045).\vspace{2mm}

\noindent\textbf{Data Availability Statements}

The data that supports the findings of this study are available within the article.\vspace{2mm}

\noindent{\bf Conflict of Interest}

The authors have no conflicts to disclose.


\begin{thebibliography}{99}


     \bibitem{nmkdv1} M.J. Ablowitz and Z.H. Musslimani,
   \newblock{Inverse scattering transform for the integrable nonlocal nonlinear Schr\"odinger equation},
   Nonlinearity, 29 (2016), 915-946.

   \bibitem{nmkdv2} M.J. Ablowitz, Z.H. Musslimani,
   \newblock{Integrable nonlocal nonlinear equations},
    Stud. Appl. Math. 139 (2017), 7-59.

     \bibitem{PT} C.M. Bender, S. Boettcher,
   \newblock{Real spectra in non-Hermitian Hamiltonians having PT symmetry},
   Phys. Rev. Lett. 80 (1998), 5243-5246.



    \bibitem{10}  S. Y. Lou and F. Huang, Alice-Bob physics: coherent solutions of nonlocal KdV systems, Sci. Rep. 7 (2017) 869.

    \bibitem{11}  X. Y. Tang, Z. F. Liang and X. Z. Hao,
   \newblock{Nonlinear waves of a nonlocal modified KdV equation in the atmospheric and oceanic dynamical system},
   Comm. Nonl. Sci. Numer. Simul. 60 (2018), 62-71.


   \bibitem{ZY} G. Zhang, Z. Yan,
    \newblock{Inverse scattering transforms and soliton solutions of focusing and
     defocusing nonlocal mKdV equations with non-zero boundary conditions},
    \newblock{\em Pys. D}, 402(2020), 132170.

    \bibitem{Darboux} J. L. Ji, Z. N. Zhu,
    \newblock{On a nonlocal modified Korteweg-de Vries equation: Integrability, Darboux transformation and soliton solutions},
    Comm. Nonl. Sci. Numer. Simul., 42 (2017), 699.

    \bibitem{IST} J. L. Ji, Z. N. Zhu,
    \newblock{Soliton solutions of an integrable nonlocal modified Korteweg-de Vries equation through inverse scattering transform},
     J. Math. Anal. Appl., 453 (2017), 973-984.

    \bibitem{HF} F. J. He, E. G. Fan and J. Xu,
    \newblock{Long-time asymptotics for the nonlocal mKdV equation},
     Comm. Theor. Phys., 71 (2019), 475-488.

    \bibitem{DZ1993} P. Deift, X. Zhou,
    \newblock{A steepest descent method for oscillatory Riemann-Hilbert prblems. Asymptotics for the MKdV equation},
    Ann. Math., 137(1993), 295-368.


    \bibitem{ZF1} X. Zhou, E. G. Fan, Long time asymptotics for the nonlocal mKdV equation
with nite density initial, Physica D: Nonlinear Phenomena, 440 (2022), 133458.

    \bibitem{ZF2}  X. Zhou, E. G. Fan, Long time asymptotic behavior for the nonlocal mKdV
equation in solitonic space-time regions, Mathematical Physics, Analysis and
Geometry, 26 (2023), 1-53.

    \bibitem{MM1} K. T. R. McLaughlin and P. D. Miller,
    \newblock{The $\bar{\partial}$-steepest descent method and the asymptotic behavior of polynomials orthogonal on the unit circle with
     fixed and exponentially varying non-analytic weights},
     Int. Math. Res. Not., (2006), Art. ID 48673.

    \bibitem{MM2} K. T. R. McLaughlin and P. D. Miller,
    \newblock{The $\bar{\partial}$-steepest descent method for orthogonal polynomials on the real line with varying weights},
    Int. Math. Res. Not., (2008), Art. ID 075

 \bibitem{fNLS}
M. Borghese, R. Jenkins, K. T. R. McLaughlin,  P. D. Miller,
\newblock { Long-time asymptotic behavior of the focusing nonlinear Schr\"odinger equation, }
\newblock {  Ann. I. H. Poincar\'e Anal.}, 35 (2018), 887-920.



\bibitem{Liu3}
R. Jenkins, J. Liu, P. Perry, C. Sulem,
\newblock Soliton resolution for the derivative nonlinear Schr\"odinger equation,
\newblock {Commun. Math. Phys.},  363 (2018), 1003-1049.


       \bibitem{LJQ}
       J. Q. Liu,
       \newblock Long-time behavior of solutions to the derivative nonlinear Schr\"odinger equation for soliton-free initial data,
       \newblock {Ann. I. H. Poincar$\acute{e}$ -Anal.}, 35 (2018), 217-265.

\bibitem{YF1} Y. L. Yang, E. G. Fan, \newblock {Soliton resolution for the short-pulse equation},
\newblock { J. Differ. Equ.}, 280 (2021), 644-689.

\bibitem{YYLmch}
Y. L. Yang,  E. G. Fan,
\newblock   On the long-time asymptotics of the modified Camassa-Holm equation in space-time solitonic regions,  Adv. Math., 402 (2022), 108340.


\bibitem{WF}   Z. Y. Wang,  E. G. Fan,  The defocusing NLS equation with nonzero background: Large-time asymptotics in the solitonless region,   J. Differ. Equ., 336 (2022), 334-373.



\bibitem{zhou1} X.  Zhou. $L^{2}$-Sobolev space bijectivity of the scattering and inverse scattering transforms,  Comm. Pure Appl. Math. 51 (1998) 697-731.


\bibitem{zhou2}   X.  Zhou, The Riemann-Hilbert problem and inverse scattering.,  SIAM J.  Math. Anal., 20(1989), 966-986.



\bibitem{Pelinovsky}   D. E. Pelinovsky, Y.  Shimabukuro,  Existence of global solutions to the derivative NLS equation with the inverse scattering transform method.
Int. Math. Res. Notices.,  18(2018), 5663-5728.





 



\end{thebibliography}
\end{document}